\documentclass[a4paper,11pt]{article}

\usepackage[utf8]{inputenc}
\usepackage[T1]{fontenc}
\usepackage{amsmath,amssymb,amsfonts,amsthm,mathtools}
\usepackage{enumitem}
\usepackage[colorlinks=true,linkcolor=blue,citecolor=blue,urlcolor=blue]{hyperref}
\usepackage[margin=1.15in]{geometry}

\newtheorem{theorem}{Theorem}[section]
\newtheorem{lemma}[theorem]{Lemma}

\newtheorem{definition}[theorem]{Definition}

\newtheorem{question}{Question}
\theoremstyle{remark}
\newtheorem{remark}[theorem]{Remark}

\newcommand{\supp}{\operatorname{supp}}
\newcommand{\Add}{\operatorname{Add}}
\newcommand{\CU}{\operatorname{CU}}

\newcommand{\SepCoded}{\operatorname{SepCoded}}
\newcommand{\Tr}{\operatorname{Tr}}
\newcommand{\forceP}{\mathbb P}
\newcommand{\forceQ}{\mathbb Q}
\newcommand{\forceR}{\mathbb R}
\newcommand{\forceJ}{\mathbb J}
\newcommand{\forceB}{\mathbb B}
\newcommand{\Sig}{\mathbf{\Sigma}}
\newcommand{\Pig}{\mathbf{\Pi}}
\newcommand{\Del}{\mathbf{\Delta}}

\newcommand{\forces}{\Vdash}
\newcommand{\restr}{\upharpoonright}

\newcommand{\vecS}{\vec S}
\newcommand{\vecC}{\vec C}

\title{Forcing $\Sig^1_1$-Separation on $\omega_1^{\omega_1}$}
\author{Stefan Hoffelner\footnote{The author's research was funded in whole by the Austrian Science Fund (FWF) Grant-DOI 10.55776/P37228.
For the purpose of open access, the author has applied a CC BY public copyright license to any Author Accepted Manuscript version arising from this submission.}}
\date{May 2026}

\begin{document}
\maketitle

\begin{abstract}
We prove that it is consistent that every two disjoint boldface
$\Sig^1_1$ subsets of $2^{\omega_1}$ can be separated by a
boldface $\Del^1_1$ set.
This is the
first construction of a universe in which the $\Sig^1_1$-separation property holds at $\omega_1$.
The forcing starts from $L$ and preserves CH
and therefore also $\omega_1^{<\omega_1}=\omega_1$.
\end{abstract}

\section{Introduction}

Generalized descriptive set theory studies the spaces $\kappa^\kappa$ and
$2^\kappa$ for uncountable regular cardinals $\kappa$.
Already for
$\kappa=\omega_1$, the structure of the definable pointclasses is much less rigid
than in the classical case.
In particular, the relations between Borel sets,
Borel$^*$ sets, $\Del^1_1$ sets and
$\Sig^1_1$ sets become sensitive to forcing and to the ambient
universe
\cite{MeklerVaananen,FriedmanHyttinenWeinstein,HyttinenKulikovBorelStar,LueckeSigma11Definability,LueckeSchlichtContinuousImages,LueckeMottoRosSchlichtHurewicz,FriedmanKhomskiiKulikovRegularity,AgostiniMottoRosSchlichtPolish,AgostiniChapmanMottoRosPittonBorel,KhomskiiLaguzziLoeweSharankouQuestions}.
One of the central regularity principles for analytic sets is separation.
In
classical descriptive set theory, any two disjoint analytic sets can be separated
by a Borel set \cite{KechrisClassicalDST,MoschovakisDST}.
In the generalized
setting, the corresponding question becomes more delicate. The usual ZFC
separation theorem gives Borel$^*$ separators for disjoint
$\Sig^1_1$ sets
\cite{MeklerVaananen,FriedmanHyttinenWeinstein,HyttinenKulikovBorelStar}, but
this leaves open whether the separator can consistently be chosen at the
$\Del^1_1$ level.
The main result of this paper is a positive consistency theorem for this stronger
separation principle.
We force a model in which the full
$\Sig^1_1$-separation property holds on $2^{\omega_1}$: whenever
\[
   A_0,A_1\subseteq 2^{\omega_1}
\]
are disjoint $\Sig^1_1$ sets, there is a
$\Del^1_1$ set $D$ such that
\[
   A_0\subseteq D
   \quad\text{and}\quad
   D\cap A_1=\emptyset .
\]
Thus the theorem is asserting a global separation principle for all analytic separation
problems appearing in the final model.
\begin{theorem}[Main theorem]
\label{thm:intro-main}
Assume $V=L$. There is a cardinal-preserving forcing extension satisfying
\[
   2^\omega=\omega_1,
   \qquad
   2^{\omega_1}=\omega_2,
\]
in which every two disjoint $\Sig^1_1$ subsets of
$2^{\omega_1}$ are separated by a $\Del^1_1$ set.
\end{theorem}

The final model lies in the usual small-base regime
$\omega_1^{<\omega_1}=\omega_1$. In this setting the standard tree-projection
presentation of $\Sig^1_1$ agrees with its
$\Sigma_1(H(\omega_2))$ presentation, and the resulting pointclass has the
expected closure and absoluteness properties used throughout the construction
\cite{LueckeSigma11Definability}.
The theorem also has a natural interpretation in connection with one of the basic
open problems about the generalized Borel hierarchy.
In the classical case
$\kappa=\omega$, the classes Borel, Borel$^*$ and
$\Del^1_1$ coincide. At uncountable cardinals the situation is
different.
In the usual context $\kappa^{<\kappa}=\kappa$, one has
\[
   \mathrm{Borel}
   \subseteq
   \Del^1_1
   \subseteq
   \mathrm{Borel}^*
   \subseteq
   \Sig^1_1 .
\]
The possible relation between $\mathrm{Borel}^*$ and
$\Del^1_1$ remains a central question.
If a model of
$\mathrm{Borel}^*=\Del^1_1$ were available, then the known
Borel$^*$ separation theorem would immediately imply
$\Sig^1_1$-separation by $\Del^1_1$ sets.
From this perspective, Theorem~\ref{thm:intro-main} proves the consistency of the
separation-theoretic consequence of the possible identity
$\mathrm{Borel}^*=\Del^1_1$. It does not decide the global
equality.
Rather, it shows that one of the most visible consequences of such an
equality can be forced directly: for analytic separation problems, Borel$^*$
separation can consistently be sharpened all the way down to
$\Del^1_1$ separation.
The construction should also be viewed as part of a broader forcing program around
the implication pattern
\[
   \text{uniformization}
   \Rightarrow
   \text{reduction}
   \Rightarrow
   \text{separation}.
\]
At the classical projective levels, this program includes forcing
$\Sig^1_3$-separation, separating
$\Pig^1_3$-reduction from
$\Pig^1_3$-uniformization, forcing
$\Pig^1_n$-uniformization, and combining or separating
uniformization, reduction and separation principles by forcing
\cite{HoffelnerSigma13Separation,HoffelnerPi13Reduction,HoffelnerPi1nUniformization,HoffelnerForcingAxiomsUniformization,HoffelnerGlobalSigmaBPFA,HoffelnerMAIFailureSeparation,HoffelnerFailureReductionPresenceSeparation,HoffelnerUpperSigmaUniformization,HoffelnerSigma13Sigma14Uniformization}.
The present paper is the generalized $2^{\omega_1}$ separation analogue of that
program.
The stationary-coding and Suslin-tree coding background also connects the present
work with results on definability of the nonstationary ideal and with projective
well-order and coding constructions
\cite{FriedmanHoffelnerNS,HoffelnerNSDeltaOne,HoffelnerLarsonSchindlerWu,HoffelnerPi13UniformizationWellorder,HoffelnerLargeContinuumGlobalSigmaWellorder,HoffelnerDelfinoLocal,HoffelnerDelfinoLocalII}.
We emphasize, however, that the argument in this paper is tailored to separation.
It does not presently force $\Pig^1_1$-reduction or
$\Pig^1_1$-uniformization on $2^{\omega_1}$, nor does it yield such
principles by a direct adaptation of the methods used in
\cite{HoffelnerPi13Reduction,HoffelnerPi13UniformizationWellorder,HoffelnerPi1nUniformization,HoffelnerUpperSigmaUniformization}.
We believe however that the methods introduced here will play an important role when one tries to generalize the result to arbitrary, regular cardinals $\kappa$.
We close the introduction with a brief description of the paper. Section
\ref{sec:preliminaries} fixes the generalized descriptive set theoretic notation.
Section~\ref{sec:E-complete} introduces $E$-complete and $E$-proper forcings and
records the preservation facts used in the construction.
Section
\ref{sec:preliminary-model} builds the Cohen reservoirs and the definable
Suslin-tree apparatus. Section~\ref{sec:direct-container-coding} defines the
direct container coding.
Section~\ref{sec:allowability} introduces clean
allowability and the separation iteration. Section~\ref{sec:correctness-new}
proves the separation theorem and verifies the final cardinal arithmetic.
The
paper ends with a short list of open questions.

\section{Generalized descriptive set theory preliminaries}\label{sec:preliminaries}

We work on $2^{\omega_1}$ with the bounded topology.
If
$s\in2^{<\omega_1}$, then
\[
   N_s=\{x\in2^{\omega_1}:s\subseteq x\}
\]
is basic open. A tree on $(2^{<\omega_1})^n$ is a downward closed set of tuples
of equal length.
If $T$ is a tree on $(2^{<\omega_1})^{n+1}$, let
\[
   p[T]=\{\vec x\in(2^{\omega_1})^n:\exists y\in2^{\omega_1}
      \ ((\vec x,y)\in[T])\}.
\]

In the final model we will have CH, hence
\[
   2^{<\omega_1}=\omega_1.
\]
Thus arbitrary trees on $2^{<\omega_1}$ have size at most $\omega_1$ and can be
coded by elements of $2^{\omega_1}$.
Consequently the standard tree-projection
presentation of $\Sigma^1_1$ agrees with the syntactic boldface presentation by
fixed lightface closed matrices and parameters from $2^{\omega_1}$.
For bookkeeping purposes we use the syntactic presentation. Fix a recursive
enumeration of the lightface closed matrices.
If $m<\omega$ is a code and
$p\in2^{\omega_1}$, write
\[
   M_m(x,p)
\]
for the corresponding $\Sigma^1_1(p)$ assertion, i.e. for a formula of the form
\[
   (\exists y\in2^{\omega_1})\,\theta_m(x,y,p),
\]
where $\theta_m$ is closed.
The boldface classes are obtained by allowing the
parameter $p\in2^{\omega_1}$. Since CH holds in the final model, this convention
is equivalent to the usual boldface tree-parametric one.
For the general
correspondence between tree projections and $\Sigma_1(H(\kappa^+))$ definitions,
including its forcing-absoluteness features, see
\cite{LueckeSigma11Definability}.
For the Borel/Borel$^*$ side of generalized
Baire space and infinitary-language variants, see
\cite{FriedmanHyttinenWeinstein,HyttinenKulikovBorelStar}.
\begin{definition}
The $\Sigma^1_1$-separation property on $2^{\omega_1}$ says that whenever
$A_0,A_1\subseteq2^{\omega_1}$ are disjoint boldface $\Sigma^1_1$ sets, there is
a boldface $\Delta^1_1$ set $D\subseteq2^{\omega_1}$ such that
\[
   A_0\subseteq D
   \quad\text{and}\quad
   D\cap A_1=\emptyset.
\]
\end{definition}

\section{$E$-complete forcings}\label{sec:E-complete}

The no-new-real part of the construction is organized around a fixed stationary
co-stationary set $E\subseteq\omega_1$.
This is the set of good heights used in
all fusion arguments.
This stationary-set relativization of properness and
club-shooting preservation is in the tradition of the standard club-shooting and
proper-forcing machinery
\cite{BaumgartnerHarringtonKleinberg,AbrahamProperForcing,MiyamotoSouslinCSIterations},
and of the earlier stationary-coding arguments
\cite{HoffelnerNSDeltaOne,HoffelnerLarsonSchindlerWu}.
\paragraph{$M$-generic decreasing sequences.}
Let $M\prec H(\Theta)$ be countable and let $\forceQ\in M$.
A decreasing
sequence
\[
   q_0\geq q_1\geq q_2\geq\cdots
\]
of conditions in $M\cap\forceQ$ is called \emph{$M$-generic} if the filter on
$M\cap\forceQ$ generated by the sequence is generic over $M$.
Equivalently,
for every dense open set $D\subseteq\forceQ$ with $D\in M$, there is
$n<\omega$ such that $q_n\in D$.
Equivalently, for every maximal antichain
$A\in M$ of $\forceQ$, there are $n<\omega$ and $a\in A\cap M$ such that
\[
   q_n\leq a.
\]

\begin{definition}[$E$-complete forcing]\label{def:E-complete}
Let $E\subseteq\omega_1$ be stationary. A forcing $\forceQ$ is
\emph{$E$-complete} if for every sufficiently large regular $\Theta$, every
countable elementary submodel
\[
   M\prec H(\Theta)
\]
with $\forceQ\in M$ and $M\cap\omega_1\in E$, and every decreasing sequence
\[
   q_0\geq q_1\geq q_2\geq\cdots
\]
of conditions in $M\cap\forceQ$ which is $M$-generic, there is a condition
$q_\omega\in\forceQ$ below all $q_n$.
\end{definition}

\begin{definition}[$E$-proper forcing]\label{def:E-proper}
Let $E\subseteq\omega_1$ be stationary. A forcing $\forceQ$ is
\emph{$E$-proper} if for every sufficiently large regular $\Theta$, every
countable elementary submodel $M\prec H(\Theta)$ with $\forceQ\in M$ and
$M\cap\omega_1\in E$, and every $q\in M\cap\forceQ$, there is a condition
$q^*\leq q$ which is $(M,\forceQ)$-generic.
\end{definition}

\begin{lemma}\label{lem:E-complete-no-reals}
If $\forceQ$ is $E$-complete, then $\forceQ$ is $E$-proper and adds no new
reals.
In particular, $\forceQ$ preserves $\omega_1$ and preserves every
stationary subset of $E$.
\end{lemma}

\begin{proof}
Let $M\prec H(\Theta)$ be countable with $\forceQ\in M$ and
$\delta=M\cap\omega_1\in E$, and let $q\in M\cap\forceQ$.
Enumerate in $M$ the
maximal antichains of $\forceQ$ which belong to $M$ as
$\langle A_n:n<\omega\rangle$.
Working inside $M$, build a decreasing sequence
\[
   q=q_0\geq q_1\geq q_2\geq\cdots
\]
as follows.
Since $A_n$ is maximal and $q_n\in M$, elementarity gives
$a_n\in A_n\cap M$ which is compatible with $q_n$.
Again by elementarity,
choose $q_{n+1}\in M\cap\forceQ$ such that
\[
   q_{n+1}\leq q_n,a_n.
\]
Then $q_{n+1}\leq a_n$, and hence the sequence meets the dense open set
generated by $A_n$. Thus the sequence is $M$-generic.
By $E$-completeness
there is a lower bound $q^*\leq q_n$ for all $n<\omega$. Then $q^*$ is
$(M,\forceQ)$-generic. Thus $\forceQ$ is $E$-proper.
The preservation of $\omega_1$ and of stationary subsets of $E$ is the usual
properness argument, restricted to the stationary collection of countable models
whose height lies in $E$.
To see that no reals are added, let $\dot r$ be a $\forceQ$-name for an element
of $2^\omega$ and let $q\in\forceQ$.
Choose $M\prec H(\Theta)$ countable with
$q,\forceQ,\dot r\in M$ and $M\cap\omega_1\in E$.
Inside $M$, build an
$M$-generic decreasing sequence
\[
   q=q_0\geq q_1\geq q_2\geq\cdots
\]
so that $q_n$ decides $\dot r(n)$.
Let $q_\omega$ be a lower bound. Then
$q_\omega$ decides every bit of $\dot r$, hence forces $\dot r$ to be a ground
model real.
\end{proof}

\begin{lemma}\label{lem:club-E-complete}
Let $A\subseteq\omega_1$ contain $E$, and let $\CU(A)$ be the forcing of closed
bounded subsets of $A$, ordered by end-extension.
Then $\CU(A)$ is
$E$-complete.
\end{lemma}

\begin{proof}
Let $M\prec H(\Theta)$ be countable with $\delta=M\cap\omega_1\in E$, and let
\[
   c_0\geq c_1\geq c_2\geq\cdots
\]
be an $M$-generic decreasing sequence in $\CU(A)\cap M$.
The union
$\bigcup_n c_n$ is a closed bounded subset of $\delta$ cofinal in $\delta$.
Since $\delta\in E\subseteq A$, the set
\[
   c_\omega=\bigcup_n c_n\cup\{\delta\}
\]
is a closed bounded subset of $A$ and is a lower bound for the sequence.
\end{proof}

\begin{lemma}\label{lem:E-complete-iteration}
Let
\[
   \langle\forceP_\alpha,\dot\forceQ_\alpha:\alpha<\gamma\rangle
\]
be a countable-support iteration. Suppose that for every $\alpha<\gamma$,
\[
   \forceP_\alpha\forces
   \text{``}\dot\forceQ_\alpha\text{ is }E\text{-complete}\text{.''}
\]
Then $\forceP_\gamma$ is $E$-complete.
In particular $\forceP_\gamma$ is
$E$-proper and adds no reals.
\end{lemma}

\begin{proof}
This is the standard fusion proof for $E$-complete forcings.
We recall the
argument to fix the form used later. Let $M\prec H(\Theta)$ be countable with
$M\cap\omega_1\in E$, and let
\[
   p_0\geq p_1\geq p_2\geq\cdots
\]
be an $M$-generic decreasing sequence in $M\cap\forceP_\gamma$.
Let
$\delta_M=M\cap\gamma$. The union of the supports of the $p_n$'s is countable
and contained in $M\cap\gamma$.
We define a lower bound $p_\omega$ on this
support by induction on the support order.
At coordinate $\alpha$, the initial
segment $p_\omega\restr\alpha$ is already a lower bound for the corresponding
initial segments, and in the $\forceP_\alpha$-extension the sequence of
$\alpha$-coordinates is an $M[G_\alpha]$-generic decreasing sequence in the
interpreted iterand.
Since the iterand is forced to be $E$-complete and
$M[G_\alpha]\cap\omega_1=M\cap\omega_1\in E$, there is a lower bound for the
coordinate sequence.
Placing these lower bounds at all coordinates in the
countable support gives a condition $p_\omega\in\forceP_\gamma$ below all
$p_n$.
\end{proof}

\begin{theorem}[Abraham--Shelah theorem, $E$-proper form]\label{thm:abraham}
Assume CH. Let
\[
   \langle \forceP_\alpha,
      \dot\forceQ_\alpha:\alpha<\gamma\rangle
\]
be a countable-support iteration with $\gamma\leq\omega_2$ such that
\[
   \forceP_\alpha\forces
   \text{``}\dot\forceQ_\alpha\text{ is }E\text{-proper and }
      |\dot\forceQ_\alpha|\leq\aleph_1\text{''}
\]
for every $\alpha<\gamma$.
Then $\forceP_\gamma$ is $E$-proper and satisfies
the $\omega_2$-chain condition. Moreover, every bounded intermediate extension
$V^{\forceP_\alpha}$, $\alpha<\omega_2$, satisfies CH.
\end{theorem}

\begin{proof}
This is the standard $E$-proper version of Shelah's preservation theorem for
countable-support iterations of size-$\aleph_1$ forcings under CH.
The proof is
the same as in the proper case presented by Abraham
\cite[Theorem~2.10 and the subsequent CH preservation theorem]{AbrahamProperForcing},
with the elementary submodels restricted to the stationary class
\[
   \{M\prec H(\Theta):M\cap\omega_1\in E\}.
\]
We use the theorem in this form.
\end{proof}

\begin{theorem}[Miyamoto preservation, $E$-proper form]
\label{thm:Eproper-Miyamoto}
Let $E\subseteq\omega_1$ be stationary, let $T$ be a Suslin tree, and let
\[
   \langle \forceP_\alpha,
      \dot\forceQ_\alpha:\alpha<\gamma\rangle
\]
be a countable-support iteration.
Suppose that for every $\alpha<\gamma$,
\[
   \forceP_\alpha\forces
   \text{``}\dot\forceQ_\alpha\text{ is }E\text{-proper and preserves }T
      \text{ as a Suslin tree}\text{.''}
\]
Then $\forceP_\gamma$ preserves $T$ as a Suslin tree.
More generally, the same conclusion holds uniformly for a fixed family of
Suslin trees, provided each iterand preserves every tree in the family.
\end{theorem}

\begin{proof}
This is the $E$-proper analogue of Miyamoto's countable-support preservation
theorem for Suslin trees. The proper case is due to Miyamoto
\cite{MiyamotoSouslinCSIterations}.
The $E$-proper formulation used here is the
standard relativization to countable elementary submodels $M$ with
$M\cap\omega_1\in E$; see the earlier preservation argument
\cite{HoffelnerNSDeltaOne}.
Equivalently, one uses the Miyamoto criterion with
``proper'' replaced by ``$E$-proper'' and checks $(M,\forceP\times T)$-genericity
only for such models $M$.
\end{proof}

\section{The preliminary model}\label{sec:preliminary-model}

We now build the model over which the separation iteration is performed.
The
preliminary forcing has three jobs: reserve the good stationary set $E$, add the
$\omega_1$-Cohen reservoirs whose initial-segment traces will be used for
bookkeeping, and add a uniformly
$\Sigma_1(H(\omega_2),\omega_1)$-definable independent sequence of
Jech-generic Suslin trees.
The background coding tools are almost-disjoint and
stationary coding, together with Jech-style Suslin-tree generics
\cite{JensenSolovay,JechSetTheory,FuchsHamkinsDegrees,HoffelnerNSDeltaOne,HoffelnerPi13UniformizationWellorder}.
The stationary coding is used only to make the
Suslin-tree apparatus definable, and it is arranged to respect $E$.
The Cohen
reservoirs remain generic auxiliary traces; their individual membership relation
is not part of the final projective decoding.
\subsection{The reserved stationary set}

Work in $L$. Fix a stationary co-stationary set
\[
   E\subseteq\omega_1
\]
from $L$, and put
\[
   E^*=\omega_1\setminus E.
\]
We shall use stationary coding only on subsets of $E^*$, so that all
club-shooting targets contain $E$.
We also fix, in $L$, a family
\[
   \langle E_\zeta:\zeta<\omega_2\rangle
\]
of stationary subsets of $E^*$ which are pairwise almost disjoint modulo the
nonstationary ideal.
In fact we choose the family so that distinct members have
only countable intersection.

Here is the standard construction.
Since $V=L$, Jensen's diamond principle
holds on every stationary subset of $\omega_1$. Fix a $\diamondsuit(E^*)$-
sequence
\[
   \langle D_\alpha:\alpha\in E^*\rangle.
\]
Thus, for every $X\subseteq\omega_1$, the set
\[
   A_X=\{\alpha\in E^*:D_\alpha=X\cap\alpha\}
\]
is stationary.
Using the canonical well-order of $L$, fix a sequence
\[
   \langle X_\zeta:\zeta<\omega_2\rangle
\]
of pairwise distinct subsets of $\omega_1$, and define
\[
   E_\zeta=A_{X_\zeta}.
\]
Then every $E_\zeta$ is stationary and contained in $E^*$. Moreover, if
$\zeta\neq\xi$, then $E_\zeta\cap E_\xi$ is countable.
Indeed, let
\[
   \beta=\min(X_\zeta\triangle X_\xi).
\]
For every $\alpha\in E^*$ with $\alpha>\beta$, we have
\[
   X_\zeta\cap\alpha\neq X_\xi\cap\alpha,
\]
and hence $\alpha$ cannot belong to both $E_\zeta$ and $E_\xi$.
Therefore
\[
   E_\zeta\cap E_\xi\subseteq \beta+1,
\]
which is countable.

Thus
\[
   \langle E_\zeta:\zeta<\omega_2\rangle
\]
is not a partition of $\omega_1\setminus E$, but rather a uniformly definable
almost disjoint stationary family inside $\omega_1\setminus E$.
This is the
family used by the stationary coding. Consequently every club-shooting target
contains $E$.
\subsection{Reservoir sets and Suslin trees}

Let
\[
   \Add_{\omega_1}(\omega_2)
\]
denote the standard forcing, computed in $L$, whose conditions are countable
partial functions from $\omega_2\times\omega_1$ to $2$, ordered by reverse
inclusion.
Let
\[
   \vecC=\langle C_\nu:\nu<\omega_2\rangle
\]
be the sequence of $\omega_1$-Cohen subsets of $\omega_1$ added by this forcing.
It is countably closed, adds no reals, and under CH has size $\omega_2$.
Fix in $L$ a bijection
\[
   \vartheta:(2^{<\omega_1})^L\longrightarrow\omega_1
\]
with branch-cofinal image, for example with
$\operatorname{lh}(s)\leq\vartheta(s)$ for all $s\in(2^{<\omega_1})^L$.
For $y\in2^{\omega_1}$ define its initial-segment trace by
\[
   \Tr(y)=\{\vartheta(y\restriction\alpha):0<\alpha<\omega_1\}.
\]
For each reservoir coordinate set
\[
   T_\nu=\{\vartheta(C_\nu\restriction\alpha):0<\alpha<\omega_1\}.
\]
The raw Cohen sets $C_\nu$ provide homogeneity; the actual coding traces are the
almost disjoint sets $T_\nu$.
Let $\mathbb J$ be Jech's forcing for adding an $\omega_1$-Suslin tree
\cite{JechSetTheory}.
We use the standard countable-support product
\[
   \forceJ_{\omega_2}=\prod_{\xi<\omega_2}^{\mathrm{cs}}\mathbb J_\xi.
\]
Let
\[
   \vecS=\langle S^i_\xi:i<2,\xi<\omega_2\rangle
\]
be the resulting sequence of pairs of generic Suslin trees, obtained by a fixed
constructible reindexing of the $\omega_2$ many coordinates.
\begin{lemma}\label{lem:Jech-finite-branches}
Let $\dot S$ be the $\mathbb J$-name for the generic Suslin tree.
For every
$0<n<\omega$, the two-step forcing
\[
   \mathbb J*\dot S^n
\]
has a $\sigma$-closed dense subset.
Consequently,
\[
   \mathbb J\forces \text{``}\dot S^n\text{ is }\omega\text{-distributive}\text{.''}
\]
\end{lemma}

\begin{proof}
Let $D_n$ consist of all pairs
\[
   (t,\langle s_0,\ldots,s_{n-1}\rangle)
\]
where $t$ is a countable normal tree of successor height and
$s_0,\ldots,s_{n-1}$ are nodes on the top level of $t$.
This is dense in
$\mathbb J*\dot S^n$: first decide the finitely many branch nodes mentioned by
the second coordinate and then end-extend the tree so that these nodes have
extensions on the new top level.
If
\[
   (t_m,\langle s^m_0,\ldots,s^m_{n-1}\rangle)\qquad(m<\omega)
\]
is decreasing in $D_n$, take the union of the $t_m$'s and add a new top level
carrying the limit nodes
\[
   s_j=\bigcup_m s^m_j\qquad(j<n).
\]
The resulting condition is a lower bound. Thus $D_n$ is $\sigma$-closed.
\end{proof}

\begin{lemma}\label{lem:CS-Jech-branch}
Assume CH.
Let $I$ be a set of cardinality at most $\omega_1$, and let
$m:I\to\omega\setminus\{0\}$ be a finite multiplicity function.
Let
\[
   \forceJ_I=\prod_{\iota\in I}^{\mathrm{cs}}\mathbb J_\iota
\]
and let $\dot S_\iota$ be the $\iota$-th generic Suslin tree.
In the
$\forceJ_I$-extension put
\[
   \dot{\forceB}(I,m)=
   \prod_{\iota\in I}^{\mathrm{cs}}
      \dot S_\iota^{\,m(\iota)}.
\]
Then
\[
   \forceJ_I*\dot{\forceB}(I,m)
\]
has a $\sigma$-closed dense subset. Consequently
\[
   \forceJ_I\forces
   \text{``}\dot{\forceB}(I,m)\text{ is }\omega\text{-distributive}\text{.''}
\]
Under CH, $\dot{\forceB}(I,m)$ has size $\aleph_1$.
\end{lemma}

\begin{proof}
A dense condition consists of a countable support $a\subseteq I$, a Jech tree
approximation $p(\iota)$ for each $\iota\in a$, and for each $\iota\in a$ an
$m(\iota)$-tuple of nodes on the top level of $p(\iota)$.
Density is obtained
coordinate by coordinate as in Lemma~\ref{lem:Jech-finite-branches}. For a
countable descending sequence, take the union of the supports and, at each
coordinate, take the union of the tree approximations and add the corresponding
limit top-level nodes.
The support remains countable, so this gives a lower
bound.

The final size estimate follows from CH: there are only $\aleph_1$ many
countable supports in $I$, and at each coordinate only $\aleph_1$ many possible
finite tuples of tree nodes.
\end{proof}

\subsection{Stationary coding respecting $E$}

Only the Suslin-tree apparatus must be uniformly definable over
$H(\omega_2)$ from $\omega_1$.
The Cohen reservoir sequence $\vec C$ is kept
as a generic auxiliary sequence and is not stationary-coded.
We code the
membership relation of $\vec S$ by shooting clubs through complements of the
stationary sets $E_\zeta$.
Since each $E_\zeta$ is disjoint from $E$, every
target set contains $E$, and therefore each club-shooting step is $E$-complete by
Lemma~\ref{lem:club-E-complete}.
We now define the preliminary forcing formally. Let
\[
   \forceP_{\rm res}=\Add_{\omega_1}^L(\omega_2)
\]
be the reservoir forcing defined above, and let
\[
   \forceP_{\rm J}=\forceJ_{\omega_2}
      =\prod_{\xi<\omega_2}^{\mathrm{cs}}\mathbb J_\xi
\]
be the countable-support product of Jech forcings, computed in $L$.
Put
\[
   \forceP_{\rm base}=\forceP_{\rm res}\times\forceP_{\rm J}.
\]
The symbols $\vec C$ and $\vec S$ are the canonical $\forceP_{\rm base}$-names
for the corresponding reservoir sequence and Suslin-tree sequence.
Since CH holds in $L$, fix an $L$-definable bijection
\[
   \theta:
   2\times\omega_2\times (2^{<\omega_1})^L
   \longrightarrow \omega_2 .
\]
In a $\forceP_{\rm base}$-extension, define
\[
   A_{\rm prep}=
      \{\theta(i,\xi,s):i<2,\ \xi<\omega_2,
          \ s\in S^i_\xi\}.
\]
Let $\dot A_{\rm prep}$ be the canonical $\forceP_{\rm base}$-name for this set.
Over the $\forceP_{\rm base}$-extension, let
\[
   \forceP_{\rm stat}=\forceP_{\rm stat}(\dot A_{\rm prep})
\]
be the countable-support iteration
\[
   \langle \forceP^\zeta_{\rm stat},
      \dot\forceR^\zeta_{\rm stat}:\zeta<\omega_2\rangle
\]
which, at stage $\zeta$, uses
\[
   \dot\forceR^\zeta_{\rm stat}= 
   \begin{cases}
      \CU(\omega_1\setminus E_\zeta),&\text{if }\zeta\in A_{\rm prep},\\
      \mathbf 1,&\text{if }\zeta\notin A_{\rm prep}.
\end{cases}
\]
Finally set
\[
   \forceP_{\rm prep}=\forceP_{\rm base}*\dot\forceP_{\rm stat}.
\]
Thus, in the preliminary extension,
\[
   \zeta\in A_{\rm prep}
   \quad\Longleftrightarrow\quad
   E_\zeta\text{ is nonstationary}.
\]
Indeed, the forward direction is forced by shooting a club through
$\omega_1\setminus E_\zeta$.
Conversely, if $\zeta\notin A_{\rm prep}$, then all
other club-shooting coordinates have targets which contain all but countably
many points of $E_\zeta$, by the almost-disjointness of the family
$\langle E_\xi:\xi<\omega_2\rangle$.
The standard countable-support
almost-disjoint stationary-coding argument therefore preserves the stationarity
of $E_\zeta$.
Let
\[
   G_{\rm prep}=G_{\rm base}*G_{\rm stat}
\]
be $\forceP_{\rm prep}$-generic over $L$, with $G_{\rm base}$ generic for
$\forceP_{\rm base}$ and $G_{\rm stat}$ generic for the stationary-coding tail.
Put
\[
   W=L[G_{\rm prep}].
\]

\begin{lemma}[Stationary coding preserves the tree apparatus]
\label{lem:stationary-coding-preserves-apparatus}
The stationary coding stage used in the definition of $W$ preserves the Suslin
apparatus added by the preliminary Jech product.
More precisely, in $W$:
\begin{enumerate}[label=\textup{(\roman*)}]
\item the sequence $\vecS$ remains independent;
\item each unused tree coordinate remains Suslin;
\item the finite off-the-generic-branches property remains true;
\item the $E$-good-level property remains true: if $\delta\in E$ and
      $\langle s_n:n<\omega\rangle$ is an increasing sequence of nodes in one
      of the trees with heights cofinal in $\delta$, then $\bigcup_n s_n$ is a
      node on level $\delta$;
and
\item the reserved set $E$ remains stationary.
\end{enumerate}
\end{lemma}

\begin{proof}
Each stationary-coding component is a club-shooting forcing $\CU(A)$ with
$E\subseteq A$.
By Lemma~\ref{lem:club-E-complete}, such a forcing is
$E$-complete. Hence the stationary-coding iteration is $E$-complete by
Lemma~\ref{lem:E-complete-iteration};
in particular it is $E$-proper, preserves
the stationarity of $E$, and adds no reals.
It remains to record the preservation of the tree apparatus.
Work in the model
after the reservoir forcing and the product of Jech forcings, but before the
stationary coding.
Fix finitely many tree coordinates and finitely many generic
branches through them.
The dense presentation from
Lemmas~\ref{lem:Jech-finite-branches} and \ref{lem:CS-Jech-branch} gives the
usual fusion lower bound by adding the model height $M\cap\omega_1$ as a new top
level of the relevant tree approximations.
Since all stationary-coding targets
contain $E$, the same lower bound simultaneously extends the club-shooting
coordinates whenever $M\cap\omega_1\in E$.
Thus the stationary-coding iteration is compatible with the Jech-tree fusion
argument.
It cannot add a cofinal branch through an unused Jech-generic tree,
and it cannot destroy the finite off-the-generic-branches property.
The same
fusion explicitly adds the limit node at every good model height in $E$, so the
$E$-good-level property is preserved as well.
The independence of $\vecS$ is
preserved for the same reason, applied to finite products of distinct tree
coordinates.
\end{proof}

\begin{lemma}\label{lem:preliminary-model}
In $W$ the following hold.
\begin{enumerate}[label=\textup{(\roman*)}]
\item $2^\omega=\omega_1$ and $2^{\omega_1}=\omega_2$.
\item The sequence $\vecS=\langle S^i_\xi:i<2,\xi<\omega_2\rangle$ is
      uniformly $\Sigma_1(H(\omega_2),\omega_1)$-definable.
The reservoir
      sequence $\vec C=\langle C_\nu:\nu<\omega_2\rangle$ is the
      $\Add_{\omega_1}^L(\omega_2)$-generic sequence added in the preliminary
      forcing, and the associated traces are $T_\nu$.
Neither
      $\vec C$ nor the individual traces $T_\nu$ are part of the stationary
      code.
\item The sequence $\vecS$ is independent: finite products of distinct members
      are Suslin.
\item Each $S^i_\xi$ has the finite off-the-generic-branches property.
\item The sequence $\vecS$ has the $E$-good-level property: whenever
      $\delta\in E$ and $\langle s_n:n<\omega\rangle$ is an increasing sequence
      of nodes in some $S^i_\xi$ with heights cofinal in $\delta$, the union
      $\bigcup_n s_n$ is a node of $S^i_\xi$ at level $\delta$.
\item The reserved set $E$ remains stationary.
\end{enumerate}
\end{lemma}

\begin{proof}
The reservoir forcing and the product of Jech forcings add no reals.
The
stationary coding uses only club-shooting forcings through sets containing $E$,
so it is $E$-complete and adds no reals.
Thus CH is preserved. The reservoir
forcing adds $\omega_2$ many distinct subsets of $\omega_1$, and the preliminary
forcing has size $\omega_2$ and the usual $\omega_2$-chain-condition/size control
under GCH in $L$, so $2^{\omega_1}=\omega_2$.
The definability of $\vecS$ follows from the stationary coding: for the fixed
coding map $\theta$, membership $s\in S^i_\xi$ is equivalent to the assertion
that $E_{\theta(i,\xi,s)}$ is nonstationary, i.e. to the
$\Sigma_1(H(\omega_2),\omega_1)$ statement that there is a club disjoint from
$E_{\theta(i,\xi,s)}$.
No corresponding definition of the Cohen reservoirs is
needed. The preservation of the Suslin apparatus, including the finite
off-branch and $E$-good-level properties, is
Lemma~\ref{lem:stationary-coding-preserves-apparatus}.
The independence and
finite off-branch properties used there come from the product Jech lemma and the
Fuchs--Hamkins off-branch theorem
\cite{FuchsHamkinsDegrees,HoffelnerPi13UniformizationWellorder}.
\end{proof}

\begin{remark}
The important point is that $E$ is not used for coding.
All stationary sets
whose stationarity is killed or preserved are subsets of $\omega_1\setminus E$.
Thus whenever a countable model has height in $E$, the usual fusion lower bound
can add that height as a new top point in the relevant club-shooting conditions
and as a new top level in the relevant Jech-tree approximations.
\end{remark}

\section{Direct container coding}\label{sec:direct-container-coding}

We now define the direct coding blocks. A container is an \(\omega_1\)-sized
subset of \(\omega_2\), partitioned into \(\omega_1\) many \(\omega_1\)-blocks.
A fresh \(\omega_1\)-Cohen reservoir set determines an almost-disjoint derived
trace, that trace chooses which blocks of the container are used, and in each
selected block the forcing writes the full characteristic function of the
relevant \(\omega_1\)-code.
Direct coding by almost-disjoint traces and
Suslin-tree branch patterns is the generalized analogue of the projective coding
machinery used in
\cite{JensenSolovay,DavidVeryAbsolute,HoffelnerSigma13Separation,HoffelnerPi13Reduction,HoffelnerPi1nUniformization,HoffelnerFailureReductionPresenceSeparation}.
\subsection{Containers and derived reservoir traces}

Fix in $L$ a partition
\[
   \langle I_\eta:\eta<\omega_2\rangle
\]
of $\omega_2$ into sets of size $\omega_1$.
Each $I_\eta$ is called a
container. Fix uniform $L$-definable bijections
\[
   e_\eta:\omega_1\times\omega_1\longrightarrow I_\eta.
\]
For $\gamma<\omega_1$ put
\[
   I_{\eta,\gamma}=\{e_\eta(\gamma,\alpha):\alpha<\omega_1\}.
\]
Thus
\[
   I_\eta=\bigcup_{\gamma<\omega_1} I_{\eta,\gamma}
\]
is partitioned into $\omega_1$ many pieces, each of size $\omega_1$.
For $i<2$, $\eta<\omega_2$, $\gamma<\omega_1$ and $\alpha<\omega_1$, set
\[
   S^i_{\eta,\gamma,\alpha}=S^i_{e_\eta(\gamma,\alpha)}.
\]
The pair
\[
   (S^0_{\eta,\gamma,\alpha},S^1_{\eta,\gamma,\alpha})
\]
is the tree pair at bit coordinate $\alpha$ inside the $\gamma$-th block of the
container $I_\eta$.
The reservoir sequence
\[
   \vec C=\langle C_\nu:\nu<\omega_2\rangle
\]
consists of the $\omega_1$-Cohen subsets of $\omega_1$ added by the preliminary
forcing.
The construction does not use $C_\nu$ itself as the set of coding
blocks.
Instead it uses the derived initial-segment trace
\[
   T_\nu=\{\vartheta(C_\nu\restriction\xi):0<\xi<\omega_1\}\subseteq\omega_1,
\]
where the bijection $\vartheta:(2^{<\omega_1})^L\to\omega_1$ was fixed in the
preliminary construction.
A use of the container $I_\eta$ chooses one reservoir
index $\nu<\omega_2$.
The derived trace $T_\nu$ determines the blocks of
$I_\eta$ in which the code is written: for every $\gamma\in T_\nu$, the whole
code is written into the block $I_{\eta,\gamma}$.
We shall use the following elementary consequence of replacing a Cohen reservoir
by the set of its coded initial segments.
\begin{lemma}\label{lem:reservoir-trace-independence}
In the preliminary model the sequence
\[
   \langle T_\nu:\nu<\omega_2\rangle
\]
consists of unbounded subsets of $\omega_1$ and is pairwise almost disjoint in
the strong sense that
\[
   T_\nu\cap T_\mu
\]
is bounded in $\omega_1$ whenever $\nu\neq\mu$.
Consequently, if
$F\subseteq\omega_2\setminus\{\nu\}$ is finite, then
\[
   T_\nu\setminus\bigcup_{\mu\in F}T_\mu
\]
is unbounded in $\omega_1$.
\end{lemma}

\begin{proof}
Each $T_\nu$ is unbounded because $\operatorname{lh}(s)\leq\vartheta(s)$ for
all $s\in(2^{<\omega_1})^L$ and the initial segments
$C_\nu\restriction\xi$ are pairwise distinct.
If $\nu\neq\mu$, then the two
Cohen sets $C_\nu$ and $C_\mu$ are distinct. Let $\delta<\omega_1$ be the first
coordinate at which they differ.
Then the only common initial segments of
$C_\nu$ and $C_\mu$ have length at most $\delta$.
Since $\vartheta$ is
injective,
\[
   T_\nu\cap T_\mu
   \subseteq
   \{\vartheta(C_\nu\restriction\xi):0<\xi\leq\delta\},
\]
which is countable, hence bounded in $\omega_1$.
The final assertion follows
because a finite union of bounded subsets of $\omega_1$ is bounded.
\end{proof}

\subsection{The tuple code}

For each separation tuple
\[
   t=(\rho,x,m,k,p,i)
\]
with $\rho<\omega_2$, $x,p\in2^{\omega_1}$, $m,k<\omega$ and $i<2$, fix a
canonical code
\[
   X_t\subseteq\omega_1.
\]
The coding map is injective and uniformly continuous in the $2^{\omega_1}$
coordinates.
For example one may interleave the characteristic functions of
$x$ and $p$ with fixed codes for $\rho,m,k,i$ using a constructible bijection
$\omega_1\cong\omega_1\times\omega$.
Only injectivity is needed in the
no-unwanted-code argument, because the derived reservoir traces are pairwise
almost disjoint.
\begin{definition}[Direct coding block]\label{def:direct-coding-block}
Let $t=(\rho,x,m,k,p,i)$ be a separation tuple, let $\eta<\omega_2$ be a
container, and let $\nu<\omega_2$ be a reservoir index.
The direct coding forcing
\[
   \forceB_{\eta,\nu,t}
\]
adds, for every $\gamma\in T_\nu$ and every $\alpha<\omega_1$, a cofinal
branch through
\[
   S^{X_t(\alpha)}_{\eta,\gamma,\alpha}.
\]
A condition is a countable partial function whose domain is a countable subset
of
\[
   \{(\gamma,\alpha):\gamma\in T_\nu,\ \alpha<\omega_1\}
\]
and which assigns to $(\gamma,\alpha)$ a node of
$S^{X_t(\alpha)}_{\eta,\gamma,\alpha}$, ordered by coordinatewise extension.
\end{definition}

More generally, if a finite list of tuples
\[
   t_0,\ldots,t_{r-1}
\]
is coded into the same container $I_\eta$, we use finitely many derived reservoir
traces
\[
   T_{\nu_0},\ldots,T_{\nu_{r-1}}
\]
and force with the countable-support product of the blocks
\[
   \forceB_{\eta,\nu_\ell,t_\ell}\qquad(\ell<r).
\]
We require the reservoir indices $\nu_\ell$ to be distinct for the distinct uses
of the same container.
The derived traces may have bounded overlap, but they
are pairwise almost disjoint by Lemma~\ref{lem:reservoir-trace-independence}.
Decoding will look for a cofinal subtrace on which no other use of the same
container is active.
\begin{lemma}\label{lem:direct-block-E-complete}
Assume CH in the ambient model and suppose the Suslin-tree apparatus satisfies
the $E$-good-level property from Lemma~\ref{lem:preliminary-model}.
Every direct
coding block $\forceB_{\eta,\nu,t}$ is $E$-complete and has size $\aleph_1$.
The same holds for any finite product of direct coding blocks using one container
and finitely many derived reservoir traces.
Hence these blocks are $E$-proper and add no
reals.
\end{lemma}

\begin{proof}
The size is $\aleph_1$ under CH: a condition has countable support in a set of
size $\aleph_1$, and at each coordinate the relevant tree has size $\aleph_1$.
We prove $E$-completeness. Let $M\prec H(\Theta)$ be countable with
$\delta=M\cap\omega_1\in E$, and let
\[
   q_0\geq q_1\geq q_2\geq\cdots
\]
be an $M$-generic descending sequence in the direct coding block.
The union of
the supports is countable. Fix a coordinate $(\gamma,\alpha)$ in this union.
The sequence of nodes assigned at this coordinate is increasing in the
corresponding Jech-generic Suslin tree.
If the heights are cofinal in $\delta$,
then by the good-level property at heights in $E$, the union of this increasing
sequence is a node on level $\delta$ of that tree.
If the heights are bounded
below $\delta$, first take the eventual union and then extend it to level
$\delta$ using normality of the tree.
Put this limit node into the coordinate.
Doing this simultaneously for all coordinates in the countable union of supports
gives a condition below every $q_n$.
The argument for finitely many traces is identical, since the union of finitely
many countable supports is countable.
\end{proof}

\begin{remark}
Lemma~\ref{lem:direct-block-E-complete} is the quotient form of the dense
$\sigma$-closed presentation in Lemma~\ref{lem:CS-Jech-branch}.
In the combined
Jech-plus-branch forcing, the lower bound is obtained by end-extending the tree
approximations and adding limit top-level nodes.
After the preliminary Jech and
stationary-coding forcing has been performed, the same construction is reflected
in the $E$-good-level property of the fixed tree sequence.
\end{remark}

\subsection{Preservation of unused tree coordinates}

The proof of the no-unwanted-code lemma uses a preservation fact about tree
coordinates which are not deliberately selected by a direct coding block.
We
isolate this fact here so that the later argument is only a decoding argument.
\begin{definition}
Let $\forceB_{\eta,\nu,t}$ be a direct coding block. Its tree-coordinate support
is
\[
   \supp_{\rm tr}(\forceB_{\eta,\nu,t})=
   \{(X_t(\alpha),e_\eta(\gamma,\alpha)):
        \gamma\in T_\nu,
        \alpha<\omega_1\}.
\]
Equivalently, $\forceB_{\eta,\nu,t}$ uses exactly the trees
\[
   S^{X_t(\alpha)}_{\eta,\gamma,\alpha}
   \qquad(\gamma\in T_\nu,\ \alpha<\omega_1).
\]
A countable-support partial run uses a tree coordinate $(i,\tau)$ if some
iterand in the run has $(i,\tau)$ in its tree-coordinate support.
\end{definition}

\begin{lemma}[Jech sealing and off-branch preservation]
\label{lem:jech-apparatus-product-preservation}
Let $K$ be a set of tree coordinates of size at most $\omega_1$, and let
$\tau=(i,\nu)$ be a tree coordinate not in $K$.
Force first with the product of
Jech forcings adding the trees indexed by $K\cup\{\tau\}$ and then add, with
countable support, the branches through the coordinates in $K$ specified by a
finite multiplicity function on each coordinate.
In the resulting extension,
the tree $S^i_\nu$ remains Suslin.

More generally, if finitely many branches through $S^i_\nu$ are deliberately
added first, then every subtree of $S^i_\nu$ off the union of those finitely many
branches remains Suslin after the remaining branch forcings on $K$.
\end{lemma}

\begin{proof}
Write
\[
   m:K\longrightarrow\omega
\]
for the finite multiplicity function, and let
\[
   \mathbb R(K,m,\tau)=
   \Bigl(\prod_{\sigma\in K\cup\{\tau\}}^{\rm cs}\mathbb J_\sigma\Bigr)
   *\dot{\mathbb B}(K,m),
\]
where
\[
   \dot{\mathbb B}(K,m)=
   \prod_{\sigma\in K}^{\rm cs}\prod_{j<m(\sigma)}\dot S_\sigma .
\]
Thus no branch coordinate over $\tau$ occurs in $\dot{\mathbb B}(K,m)$.
It is
enough to prove
\[
   \mathbb R(K,m,\tau)\forces
   \text{``}S^i_\nu\text{ is Suslin}\text{.''}
\]
Fix $p\in\mathbb R(K,m,\tau)$ and a name $\dot A$ such that
\[
   p\forces \text{``}\dot A\subseteq S^i_\nu
        \text{ is a maximal antichain}\text{.''}
\]
Choose a countable
\[
   M\prec H(\Theta)\quad\text{with}\quad
   p,\dot A,K,m,\tau,\mathbb R(K,m,\tau)\in M,
\]
where $\Theta$ is large, and put $\delta=M\cap\omega_1$.
Work in the dense
$\sigma$-closed presentation from Lemma~\ref{lem:CS-Jech-branch}. In this
presentation the coordinates in $K$ carry their finitely many branch markers,
whereas the coordinate $\tau$ carries only the Jech-tree approximation.
Construct an $M$-generic descending sequence
\[
   p=p_0\geq p_1\geq p_2\geq\cdots
\]
which meets all dense subsets of the presentation belonging to $M$ and, for the
$\tau$-coordinate, meets the usual sealing requirements for $\dot A$.
The latter
requirements are the dense sets ensuring that, whenever a possible future
$\delta$-level node $u$ at the $\tau$-coordinate is produced by the fusion, some
condition in the sequence decides some ground-model node
$a\in S^i_\nu\cap M$ below $u$ to be a member of $\dot A$.
Let $p_\omega$ be the fusion of this sequence. Concretely, for
$\sigma\in K$ we take the unions of the tree approximations and add the limit
nodes determined by the finitely many branch markers.
At the unused coordinate
$\tau$ we take the union of the Jech approximations and add the sealing level
$L_\tau\subseteq (S^i_\nu)_\delta$.
The construction gives
\[
   p_\omega\forces
   (\forall u\in L_\tau)(\exists a\in S^i_\nu\cap M)
      \, (a\in\dot A\ \wedge\ a<_{S^i_\nu}u) .
\tag{$*$}
\]
Every node of $S^i_\nu$ of height at least $\delta$ extends a unique member of
$L_\tau$.
Hence $p_\omega$ forces that no element of $\dot A$ has height
$>\delta$: if $b\in\dot A$ and $\operatorname{ht}(b)>\delta$, then
$b\restriction\delta\in L_\tau$, and by $(*)$ there is $a\in S^i_\nu\cap M$ such that
$a\in\dot A$ and
\[
   a<_{S^i_\nu} b\restriction\delta <_{S^i_\nu} b,
\]
contradicting that $\dot A$ is an antichain.
Thus
\[
   p_\omega\forces
   \dot A\subseteq S^i_\nu\restriction(\delta+1).
\]
The set $S^i_\nu\restriction(\delta+1)$ is countable.
Since $p$ and $\dot A$
were arbitrary,
\[
   \mathbb R(K,m,\tau)\forces
   \text{``}S^i_\nu\text{ has no uncountable antichain}\text{.''}
\]
Thus $S^i_\nu$ remains Suslin.
For the off-branch statement, suppose first that $n<\omega$ branches
\[
   \dot b_0,\ldots,\dot b_{n-1}
\]
through $S^i_\nu$ are deliberately added.
By the Fuchs--Hamkins off-branch
preservation theorem
\cite{FuchsHamkinsDegrees,HoffelnerPi13UniformizationWellorder},
\[
   \mathbb J_\tau*\dot S_\tau^{\,n}
   \forces
   \text{``every subtree of }S^i_\nu\setminus
   \bigcup_{j<n}\dot b_j\text{ is Suslin.''}
\]
Now work in the extension by these $n$ branches and repeat the argument above
with a name $\dot U$ for an off-branch subtree and a name $\dot A$ for a maximal
antichain of $\dot U$.
The fusion is carried out exactly as before, except that
the sealing level is chosen inside $\dot U$.
It yields
\[
   p_\omega\forces
   \dot A\subseteq \dot U\restriction(\delta+1),
\]
so the later branch forcings on the coordinates in $K$ preserve the Suslinity of
$\dot U$.
This proves the stated off-branch preservation.
\end{proof}

\begin{lemma}
\label{lem:atomic-unused-coordinate-preservation}
Let $\forceB_{\eta,\nu,t}$ be a direct coding block and let $(i,\tau)$ be a tree
coordinate which is not in $\supp_{\rm tr}(\forceB_{\eta,\nu,t})$.
Then
$\forceB_{\eta,\nu,t}$ preserves the Suslinity of $S^i_\tau$.

The same holds for any finite product of direct coding blocks, provided
$(i,\tau)$ is outside the union of their tree-coordinate supports.
If finitely
many branches through $S^i_\tau$ have already been intentionally added, then the
finite product preserves the Suslinity of every subtree of $S^i_\tau$ off those
branches.
\end{lemma}

\begin{proof}
Let
\[
   K_0=\supp_{\rm tr}(\forceB_{\eta,\nu,t}).
\]
Since $(i,\tau)\notin K_0$, the pre-quotient forcing for this block has the
form
\[
   \Bigl(\prod_{\sigma\in K_0\cup\{(i,\tau)\}}^{\rm cs}\mathbb J_\sigma\Bigr)
   *\dot{\mathbb B}(K_0,1),
\]
where
\[
   \dot{\mathbb B}(K_0,1)=
   \prod_{\sigma\in K_0}^{\rm cs}\dot S_\sigma .
\]
The branch part is precisely the direct block $\forceB_{\eta,\nu,t}$ after
passing to the quotient over the preliminary Jech and stationary-coding generics.
By Lemma~\ref{lem:jech-apparatus-product-preservation}, the pre-quotient
forcing satisfies
\[
   \forces \text{``}S^i_\tau\text{ is Suslin}\text{.''}
\]
The quotient forcing theorem therefore gives, in the preliminary model $W$,
\[
   \forceB_{\eta,\nu,t}\forces
   \text{``}S^i_\tau\text{ is Suslin}\text{.''}
\]
Equivalently, if some $b\in\forceB_{\eta,\nu,t}$ forced a name $\dot A$ to be an
uncountable antichain of $S^i_\tau$, then the corresponding pre-quotient
condition would contradict the displayed forcing assertion.
Now let
\[
   \forceB_*=\prod_{\ell<r}^{\rm cs}\forceB_{\eta,\nu_\ell,t_\ell}
\]
be a finite product of direct coding blocks and put
\[
   K_*=\bigcup_{\ell<r}
        \supp_{\rm tr}(\forceB_{\eta,\nu_\ell,t_\ell}).
\]
For $\sigma\in K_*$ define the finite multiplicity
\[
   m_*(\sigma)=
   \bigl|\{\ell<r:\sigma\in
      \supp_{\rm tr}(\forceB_{\eta,\nu_\ell,t_\ell})\}\bigr|
.
\]
Then $\forceB_*$ is the quotient interpretation of
$\dot{\mathbb B}(K_*,m_*)$. If $(i,\tau)\notin K_*$, another application of
Lemma~\ref{lem:jech-apparatus-product-preservation} yields
\[
   \forceB_*\forces
   \text{``}S^i_\tau\text{ is Suslin}\text{.''}
\]
The off-branch version is obtained by applying the off-branch part of
Lemma~\ref{lem:jech-apparatus-product-preservation} to the same pair
$(K_*,m_*)$: after the finitely many intentional branches through $S^i_\tau$ are
factored off, the quotient product on $K_*$ preserves every subtree disjoint from
their union.
\end{proof}

\begin{lemma}
\label{lem:iterated-unused-coordinate-preservation}
Let
\[
   \langle\forceP_\alpha,\dot\forceQ_\alpha:\alpha<\gamma\rangle
\]
be a clean countable-support partial run of direct coding blocks, with
$\gamma\leq\omega_2$.
Suppose that a tree coordinate $(i,\nu)$ is not used by
any iterand of the run.
Then
\[
   \forceP_\gamma\forces
   \text{``}S^i_\nu\text{ is Suslin}\text{.''}
\]

More generally, suppose that the run deliberately adds only finitely many
branches through $S^i_\nu$.
Then, after those branches are added, the remaining
iteration preserves the Suslinity of every subtree of $S^i_\nu$ off their union.
\end{lemma}

\begin{proof}
By Lemma~\ref{lem:atomic-unused-coordinate-preservation}, each atomic direct
coding block which avoids $(i,\nu)$ preserves $S^i_\nu$.
These iterands are
$E$-proper and have size $\aleph_1$ by
Lemmas~\ref{lem:direct-block-E-complete} and \ref{lem:E-complete-no-reals}.
Therefore Theorem~\ref{thm:Eproper-Miyamoto} implies that the whole partial run
preserves the Suslinity of $S^i_\nu$.
For the off-branch version, first factor off the finitely many intentional
branches through $S^i_\nu$.
Lemma~\ref{lem:atomic-unused-coordinate-preservation}
says that every later atomic block preserves the relevant off-branch subtree,
and Theorem~\ref{thm:Eproper-Miyamoto} again lifts this preservation through the
countable-support iteration.
\end{proof}

\subsection{Separation codes}

\begin{definition}[Direct separation code]\label{def:direct-sepcode}
Fix a container $\eta<\omega_2$. We write
\[
   \SepCoded^\eta_\rho(x,m,k,p,i)
\]
if there are a strictly increasing function
$d:\omega_1\to\omega_1$ and an $\omega_1$-length code for a transitive model
$N$ such that $N$ verifies the following.
\begin{enumerate}[label=\textup{(\roman*)}]
\item The sequence $\vecS$ and the fixed container/block indexing are decoded
      according to the fixed $\Sigma_1(H(\omega_2),\omega_1)$ definitions.
\item The tuple $t=(\rho,x,m,k,p,i)$ has canonical bit code
      $X_t\subseteq\omega_1$.
\item The function $d$ is strictly increasing. Hence its range is cofinal in
      $\omega_1$.
\item For every $\beta<\omega_1$ and every $\alpha<\omega_1$, $N$ contains a
      cofinal branch through
      \[
         S^{X_t(\alpha)}_{\eta,d(\beta),\alpha}.
\]
\end{enumerate}
We write $\SepCoded_\rho(x,m,k,p,i)$ if
$\SepCoded^\eta_\rho(x,m,k,p,i)$ holds for some $\eta<\omega_2$.

If $B\subseteq\omega_2$ is a set of containers to be ignored, we write
\[
   \SepCoded^{B}_\rho(x,m,k,p,i)
\]
if $\SepCoded^\eta_\rho(x,m,k,p,i)$ holds for some container
$\eta\notin B$.
In the final separator, $B$ will be the activation-stage error
term $B_\rho$, namely the set of all containers already used before the pair
activated.
Thus the decoder explicitly ignores all pre-activation coding noise.
\end{definition}

\begin{lemma}\label{lem:direct-sepcode-sigma11}
For each $\rho<\omega_2$ and each parameter
$B\in H(\omega_2)$ with $B\subseteq\omega_2$, the relation
\[
   \SepCoded^B_\rho(x,m,k,p,i)
\]
is $\Sigma^1_1$, uniformly in $m,k,p,i$ and in a code for $B$.
In particular,
$\SepCoded_\rho(x,m,k,p,i)$ is $\Sigma^1_1$.
\end{lemma}

\begin{proof}
The witness is a subset of $\omega_1$ coding the container $\eta$, the increasing
cofinal trace $d$, and a well-founded extensional structure whose transitive
collapse is the model $N$.
The additional requirement for
$\SepCoded^B_\rho$ is the bounded check $\eta\notin B$, using $B$ as a boldface
parameter.
Since $B\in H(\omega_2)$, this is an allowable parameter in the
$\Sigma_1(H(\omega_2))$ presentation of boldface $\Sigma^1_1$.
The verification inside $N$ is bounded: the preliminary tree apparatus $\vecS$
and the fixed container/block indexing are
$\Sigma_1(H(\omega_2),\omega_1)$-definable, and the branch assertions are
checked by the presence of the corresponding $\omega_1$-sequences of tree nodes
in $N$.
The function $d$ is an arbitrary cofinal trace; the definition does
not require $d$ to be a subtrace of any particular derived reservoir trace.
Well-foundedness and extensionality are coded in the usual closed way by an
$\omega_1$-length witness.
Thus the whole relation is a projection of a
closed condition on an $\omega_1$-code.
\end{proof}

\begin{lemma}
\label{lem:intentional-direct-codes-decoded}
Suppose a clean finite use of a container $I_\eta$ codes the tuples
$t_0,\ldots,t_{r-1}$ using pairwise distinct reservoir indices
$\nu_0,\ldots,\nu_{r-1}$, and hence the pairwise almost-disjoint traces
$T_{\nu_0},\ldots,T_{\nu_{r-1}}$.
Let $B\subseteq\omega_2$ with $\eta\notin B$.
Then, after forcing with these direct coding blocks, each $t_\ell$ satisfies its
corresponding direct separation code relative to $B$.
\end{lemma}

\begin{proof}
Fix $\ell<r$. By Lemma~\ref{lem:reservoir-trace-independence}, the set
\[
   T_{\nu_\ell}\setminus\bigcup_{j<r,\,j\neq\ell}T_{\nu_j}
\]
is unbounded in $\omega_1$.
Let $d:\omega_1\to\omega_1$ be a strictly increasing
enumeration of an unbounded subset of this set.
Along the range of $d$, no other
use of the container is active.
The direct coding block for $t_\ell$ therefore
adds the required branches through
\[
   S^{X_{t_\ell}(\alpha)}_{\eta,d(\beta),\alpha}
\]
for every $\beta,\alpha<\omega_1$.
Since $\eta\notin B$, taking $N$ to be the transitive collapse of a sufficiently
large initial segment containing the relevant apparatus, the branches just
added, and the function $d$ gives a witness to the direct separation code
relative to $B$.
\end{proof}

\subsection{No unwanted direct codes}

\begin{lemma}[No unwanted direct codes]\label{lem:no-unwanted-direct-codes}
Let $\forceP$ be a clean countable-support partial run of direct coding blocks,
and let $G\subseteq\forceP$ be generic.
Fix a container $\eta<\omega_2$. If
$W[G]$ satisfies
\[
   \SepCoded^\eta_\rho(x,m,k,p,i),
\]
then $\eta$ was intentionally used in the run to code the tuple
$(\rho,x,m,k,p,i)$.
Consequently, if $B\subseteq\omega_2$ and
$\SepCoded^B_\rho(x,m,k,p,i)$ holds, then it is witnessed by an intentional code
in some container outside $B$.
\end{lemma}

\begin{proof}
Let
\[
   t=(\rho,x,m,k,p,i)
\]
and let $X_t$ be its canonical code. Let $d:\omega_1\to\omega_1$ be the
cofinal trace appearing in the alleged code.
Suppose first that the container $\eta$ was not intentionally used. Choose any
$\beta<\omega_1$ and any $\alpha<\omega_1$.
The alleged code requires a branch
through
\[
   S^{X_t(\alpha)}_{\eta,d(\beta),\alpha}.
\]
This tree coordinate is unused by the partial run.
By
Lemma~\ref{lem:iterated-unused-coordinate-preservation}, it remains Suslin after
forcing with $\forceP$. Therefore no cofinal branch through it can appear in
$W[G]$, contradiction.
Now suppose that $\eta$ was intentionally used finitely many times, coding
\[
   t_0,\ldots,t_{r-1}
\]
with reservoir indices
\[
   \nu_0,\ldots,\nu_{r-1}
\]
and derived traces $T_{\nu_0},\ldots,T_{\nu_{r-1}}$.
If some value
$d(\beta)$ is outside
\[
   \bigcup_{\ell<r}T_{\nu_\ell},
\]
then the same unused-coordinate argument gives a contradiction.
Hence the range
of $d$ is covered by this finite union of derived traces.
Since $d$ is cofinal,
there is some $\ell<r$ such that
\[
   \{\beta<\omega_1:d(\beta)\in T_{\nu_\ell}\}
\]
is unbounded.
By the almost-disjointness of the derived traces, we may choose
$\beta<\omega_1$ such that
\[
   d(\beta)\in T_{\nu_\ell}\setminus
   \bigcup_{j<r,\,j\neq\ell}T_{\nu_j}.
\]
Thus, at the block $d(\beta)$, the only intentional use of the container is the
code for $t_\ell$.
If $t_\ell\neq t$, choose $\alpha<\omega_1$ such that
\[
   X_{t_\ell}(\alpha)\neq X_t(\alpha).
\]
The alleged code for $t$ requires a branch through
\[
   S^{X_t(\alpha)}_{\eta,d(\beta),\alpha}.
\]
At this coordinate the unique intentional active code, namely the code for
$t_\ell$, uses the distinct tree
\[
   S^{X_{t_\ell}(\alpha)}_{\eta,d(\beta),\alpha}.
\]
Hence the tree required by the alleged code is unused by the partial run.
By
Lemma~\ref{lem:iterated-unused-coordinate-preservation}, it remains Suslin after
forcing with $\forceP$, contradicting the existence of the alleged branch.
Therefore $t_\ell=t$, and the alleged code is one of the intentional codes.
\end{proof}

\begin{remark}
Finite reuse of a container is handled by different derived reservoir traces,
not by spatially disjoint slots.
The derived traces may have countable
intersections, but they are pairwise almost disjoint.
Therefore each intentional
use has cofinally many clean blocks, and the no-unwanted-code argument can thin
any alleged cofinal trace to one of these clean blocks.
This is why the
definition of $\SepCoded$ needs only positive branch information.
\end{remark}

\section{Clean allowability and the separation iteration}\label{sec:allowability}

We now define the separation iteration.
Its atomic coding blocks are the direct
$E$-complete container codings introduced above, and the forcing is arranged to
preserve CH.
The activation and allowability terminology is a generalized
Baire-space version of the allowable iterations used in the classical
projective separation/reduction constructions
\cite{HoffelnerSigma13Separation,HoffelnerPi13Reduction,HoffelnerFailureReductionPresenceSeparation}
\subsection{Sealed containers}

\begin{definition}[Sealed containers]\label{def:sealed-containers}
During the main iteration we maintain an increasing sequence
\[
   \langle B_\alpha:\alpha<\omega_2\rangle
\]
of subsets of $\omega_2$.
The set $B_\alpha$ consists of the containers used
before stage $\alpha$; its elements are sealed.
A coding block inserted at
stage $\alpha$, and every coding block occurring inside an allowable forcing
inserted at stage $\alpha$, must use only containers outside $B_\alpha$.
If the
stage uses the set $E_\alpha^{\rm cont}$ of containers, then
\[
   B_{\alpha+1}=B_\alpha\cup E_\alpha^{\rm cont}.
\]
At limit stages take unions.
\end{definition}

\begin{lemma}\label{lem:sealed-size-new}
For every $\alpha<\omega_2$, $|B_\alpha|<\omega_2$.
\end{lemma}

\begin{proof}
Each stage inserts a forcing with a presentation in $H(\omega_2)$ and hence uses
fewer than $\omega_2$ many containers.
Since $\omega_2$ is regular, the union
of $<\omega_2$ many such sets still has size $<\omega_2$.
\end{proof}

\begin{definition}[Clean allowability]\label{def:clean-run-new}
A legal partial run is a countable-support iteration of direct coding blocks,
together with the bookkeeping data for the containers and derived reservoir traces used by
those blocks, which satisfies the syntactic requirements imposed on such blocks in
Section~\ref{sec:direct-container-coding}.
Let $B\subseteq\omega_2$. A legal partial run is
\emph{$B$-clean} if it avoids every container in $B$ and uses each container only
finitely many times, with pairwise distinct derived reservoir traces for the distinct uses
of that container.
A \emph{clean iteration} means a countable-support legal
partial run which is $B$-clean for the relevant sealed set $B$;
when $B$ is not
displayed, it is the sealed set fixed by the surrounding construction.
In the
actual iteration, used containers are sealed, so a container is not reused by later
stages.
\end{definition}

\subsection{Allowability}

\begin{definition}[$0$-allowable]\label{def:zero-allowable-new}
A forcing is $0$-allowable if it has a small presentation as a countable-support
clean iteration of direct coding blocks, in the sense of
Definition~\ref{def:clean-run-new} with no external sealed set.
Small means that
the presentation belongs to $H(\omega_2)$.
\end{definition}

\begin{lemma}\label{lem:zero-allowable-E-complete}
Every $0$-allowable forcing is $E$-complete.
If CH holds in the ambient model,
every atomic direct coding block has size $\aleph_1$.
\end{lemma}

\begin{proof}
Atomic direct coding blocks are $E$-complete by
Lemma~\ref{lem:direct-block-E-complete}. Countable-support iterations of
$E$-complete forcings are $E$-complete by Lemma~\ref{lem:E-complete-iteration}.
The size statement is part of Lemma~\ref{lem:direct-block-E-complete}.
\end{proof}

Suppose a pair $(m,k,p)$ is considered at stage $\rho$ and has not been
neutralized.
It becomes active at $\rho$ if no current $B_\rho$-clean allowable
forcing over $W[G_\rho]$ can force
\[
   \exists z\,\bigl(M_m(z,p)\wedge M_k(z,p)\bigr).
\]
At activation we record the sealed set $B_\rho$ and work from then on inside the
subclass of allowable forcings avoiding $B_\rho$ and obeying the side-placement
rule for this pair.
This set $B_\rho$ is also the error term for the final
separator: containers in $B_\rho$ may contain arbitrary pre-activation branch
noise and are ignored by the decoding relation used to define the separator.
This sealed class is denoted
\[
   \Gamma^*_{\rho}(m,k,p).
\]

\begin{definition}[Side-placement rule]\label{def:side-placement-new}
Let $(m,k,p)$ activate at stage $\rho$.
At a later stage $\beta\geq\rho$, if
the bookkeeping presents $x$, then:
\begin{enumerate}[label=\textup{(\roman*)}]
\item if some $B_\beta$-clean member of $\Gamma^*_{\rho}(m,k,p)$ forces
      $M_m(x,p)$, the construction directly codes
      $(\rho,x,m,k,p,0)$ into a fresh container outside $B_\beta$;
\item otherwise it directly codes $(\rho,x,m,k,p,1)$ into a fresh container
      outside $B_\beta$.
\end{enumerate}
The side-one clause is a default placement.
\end{definition}

\begin{lemma}\label{lem:concat-new}
The allowable classes, and the sealed classes $\Gamma^*_{\rho}(m,k,p)$, are
closed under countable-support concatenation of legal clean partial runs.
\end{lemma}

\begin{proof}
Concatenating presentations gives another legal presentation. Requirements
already imposed on each piece are still obeyed in the concatenation.
If both
pieces avoid a sealed set $B$, then so does the concatenation.
New names which
appear only after the concatenation need not have been handled before they are
presented by the bookkeeping;
this is the usual partial-run convention.
\end{proof}

\subsection{The global iteration}

Fix a bookkeeping function
\[
   F:\omega_2\to H(\omega_2)
\]
which lists all relevant names for triples $(x,p,m,k)$ unboundedly often.
We
build a countable-support iteration
\[
   \langle\forceP_\alpha,\dot\forceQ_\alpha:\alpha<\omega_2\rangle
\]
over $W$, simultaneously with the sealed-container sequence
$\langle B_\alpha:\alpha<\omega_2\rangle$.
At stage $\alpha$, if $F(\alpha)$ is not a well-formed name for a tuple
$(x,p,m,k)$, the iterand is trivial.
Otherwise evaluate the tuple in
$W[G_\alpha]$.

If the pair $(m,k,p)$ has already been neutralized, do nothing.
If the current
model already satisfies
\[
   \exists z\,\bigl(M_m(z,p)\wedge M_k(z,p)\bigr),
\]
declare the pair neutralized.
If some current $B_\alpha$-clean allowable forcing
can force this intersection, insert the $<_L$-least such forcing and declare the
pair neutralized.
If no such forcing exists, activate the pair if it is not yet active. Let
$\rho$ be its activation stage.
Then handle the present point $x$ according to
Definition~\ref{def:side-placement-new}, using a fresh container outside
$B_\alpha$ and sealing it after use.
\begin{lemma}\label{lem:iteration-preservation-new}
The final forcing $\forceP_{\omega_2}$ is $E$-complete, $E$-proper, adds no reals,
and satisfies the $\omega_2$-chain condition.
Consequently the full extension
preserves cardinals and satisfies
\[
   2^\omega=\omega_1,
   \qquad
   2^{\omega_1}=\omega_2.
\]
\end{lemma}

\begin{proof}
Every atomic direct coding block is $E$-complete and, at every bounded stage,
has size $\aleph_1$.
The bounded-stage CH needed for the size computation is
provided by Theorem~\ref{thm:abraham}, since the iterands are $E$-proper of size
$\aleph_1$.
The global iteration is $E$-complete by
Lemma~\ref{lem:E-complete-iteration}, hence adds no reals by
Lemma~\ref{lem:E-complete-no-reals}. Therefore CH holds in the final model.
The same Abraham--Shelah theorem, Theorem~\ref{thm:abraham}, gives the
$\omega_2$-c.c. for the countable-support iteration of length $\omega_2$ of
$E$-proper size-$\aleph_1$ iterands.
Thus cardinals are preserved. The preliminary
forcing added $\omega_2$ many $\omega_1$-Cohen subsets of $\omega_1$, so
$2^{\omega_1}\geq\omega_2$.
The total forcing has size $\omega_2$ and satisfies
the $\omega_2$-c.c.; using GCH in $L$ and CH in the final model, the usual name
counting gives $2^{\omega_1}\leq\omega_2$.
\end{proof}

\begin{lemma}\label{lem:bounded-appearance-new}
Let $G\subseteq\forceP_{\omega_2}$ be generic over $W$. Every element of
$(2^{\omega_1})^{W[G]}$ belongs to some bounded intermediate extension
$W[G_\delta]$, $\delta<\omega_2$.
Consequently, if $M_m(x,p)$ holds in the
final model, then it already holds with witnesses in some bounded intermediate
model.
\end{lemma}

\begin{proof}
Let $\dot x$ be a name for an element of $2^{\omega_1}$. For each
$\xi<\omega_1$, choose a maximal antichain deciding $\dot x(\xi)$.
The final
forcing is $\omega_2$-c.c., so each antichain has size at most $\omega_1$.
There
are $\omega_1$ many coordinates, and every condition has countable support.
Thus the union of the supports used by all these antichains has size at most
$\omega_1$, hence is bounded in $\omega_2$.
The name is therefore a
$\forceP_\delta$-name for some $\delta<\omega_2$.

If $M_m(x,p)$ holds, choose the $\Sigma^1_1$ witness $y\in2^{\omega_1}$ and apply
the previous paragraph to $x,p,y$.
Since the matrix $\theta_m$ is closed, its
truth is absolute once the relevant objects are present.
\end{proof}

\begin{lemma}\label{lem:tails-allowable-new}
If $(m,k,p)$ activates at stage $\rho$, then every later tail
$\forceP_\beta/G_\rho$, $\rho<\beta\leq\omega_2$, is a $B_\rho$-clean legal
partial run, in the sense of Definition~\ref{def:clean-run-new}, and belongs to
the sealed class $\Gamma^*_{\rho}(m,k,p)$.
\end{lemma}

\begin{proof}
After activation, all later stages obey the recorded side-placement requirement
for the pair.
All containers in $B_\rho$ remain sealed forever, so later tails
avoid them. Later requirements only shrink the class further.
Hence the tail
is a member of the sealed class attached at $\rho$.
\end{proof}

\subsection{Relocation by Cohen-reservoir homogeneity}\label{subsec:cohen-reservoir-relocation}

We next record the relocation principle used in the side-placement argument.
The
point is not that the branch-adding forcings attached to different Suslin-tree
coordinates are literally isomorphic. What we use is weaker.
The preliminary
reservoir forcing contains \(\omega_2\) many \(\omega_1\)-Cohen traces, and the
existence of an allowable witness is already forced by a small condition in this
homogeneous product.
By permuting the reservoir coordinates we can make such a
condition speak about fresh resources, and we can arrange that the shifted
condition still belongs to the fixed reservoir generic.
No definability of the
individual reservoirs is used here; only the product homogeneity and the fact
that all relevant names have support of size $<\omega_2$ are used.
Let
\[
   \mathbb C=\Add_{\omega_1}(\omega_2)^L
\]
be the countable-support product used in the preliminary construction to add the
reservoir sequence
\[
   \vec C=\langle C_\nu:\nu<\omega_2\rangle.
\]
Let \(G_C\subseteq\mathbb C\) denote the fixed reservoir generic. If
\(c\in\mathbb C\), write
\[
   \supp(c)=\{\nu<\omega_2:(\exists\xi<\omega_1)\,c(\nu,\xi)
             \text{ is defined}\}.
\]
Thus \(\supp(c)\) is countable.

For a clean allowable forcing \(\mathbb Q\), fix once and for all a small
presentation of \(\mathbb Q\) as a countable-support iteration of direct coding
blocks.
Let
\[
   \operatorname{Cont}(\mathbb Q)
\]
be the set of container indices occurring in this presentation, and let
\[
   \operatorname{Tr}(\mathbb Q)
\]
be the set of reservoir-trace indices occurring in it.
Since the presentation
belongs to \(H(\omega_2)\), both sets have cardinality \(<\omega_2\).
Similarly, for the actual iteration, let
\[
   R_\alpha
\]
denote the set of reservoir-trace indices used before stage \(\alpha\).
\begin{lemma}\label{lem:trace-support-small}
For every \(\alpha<\omega_2\),
\[
   |R_\alpha|<\omega_2.
\]
\end{lemma}

\begin{proof}
At each stage the iterand has a small presentation and therefore uses fewer than
\(\omega_2\) many derived reservoir traces.
Since \(\omega_2\) is regular, the union of
\(<\omega_2\) many such sets still has cardinality \(<\omega_2\).
\end{proof}

If \(\pi\) is a permutation of \(\omega_2\), let
\[
   \widehat\pi:\mathbb C\to\mathbb C
\]
be the induced automorphism of the reservoir product:
\[
   \widehat\pi(c)(\pi(\nu),\xi)=c(\nu,\xi).
\]
We extend \(\widehat\pi\) recursively to \(\mathbb C\)-names in the usual way.
\begin{lemma}\label{lem:cohen-reservoir-overspill}
Let \(c\in\mathbb C\), and let
\[
   A_0,A_1,R\subseteq\omega_2
\]
have cardinality \(<\omega_2\), with \(A_0\cap A_1=\emptyset\).
Think of
\(A_0\) as the set of reservoir coordinates which must be fixed, \(A_1\) as the
set of coordinates supporting a witness, and \(R\) as the set of forbidden
coordinates.
Then below \(c\) it is dense to find a condition \(q\) and a permutation
\(\pi\) of \(\omega_2\) such that
\[
   \pi\restriction A_0=\operatorname{id},
   \qquad
   \pi[A_1]\cap R=\emptyset,
\]
and
\[
   q\leq c,\qquad q\leq\widehat\pi(c).
\]
Consequently, if \(c\in G_C\), then there is such a permutation \(\pi\) with
\[
   \widehat\pi(c)\in G_C.
\]
\end{lemma}

\begin{proof}
Let \(r\leq c\) be arbitrary. Since
\[
   |A_0\cup A_1\cup R\cup\supp(r)|<\omega_2
\]
and \(\omega_2\) is regular, we may choose a permutation \(\pi\) of \(\omega_2\)
which fixes \(A_0\) pointwise and sends
\[
   A_1\cup(\supp(c)\setminus A_0)
\]
away from
\[
   R\cup\supp(r)
\]
except for the part of \(\supp(c)\) already lying in \(A_0\).
This is possible
because all mentioned sets have size \(<\omega_2\).

The conditions \(r\) and \(\widehat\pi(c)\) are compatible.
On coordinates in
\(A_0\), the condition \(\widehat\pi(c)\) agrees with \(c\), and \(r\leq c\).
On
coordinates outside \(A_0\), the support of \(\widehat\pi(c)\) was moved away
from \(\supp(r)\). Let \(q\) be a common extension of \(r\) and
\(\widehat\pi(c)\).
Then \(q\leq r\leq c\) and \(q\leq\widehat\pi(c)\). This
proves density below \(c\).
If \(c\in G_C\), the generic filter meets this dense set below \(c\).
Thus for
some \(q\in G_C\) and some such \(\pi\), we have \(q\leq\widehat\pi(c)\). Since
filters are upward closed, \(\widehat\pi(c)\in G_C\).
\end{proof}

The next lemma is the exact relocation statement used below. It deliberately
asserts only the existence of a fresh allowable witness.
It does not assert that
an already chosen branch-adding forcing is isomorphic to its relocated copy.
The
``support'' of a witness refers to the reservoir coordinates on which the small
presentation of the witness, including its auxiliary choices of traces and
containers, depends.
\begin{lemma}\label{lem:fresh-resource-relocation}
Work in an intermediate model containing the reservoir generic \(G_C\).
Let
\[
   B,R\subseteq\omega_2
\]
be forbidden sets of container and reservoir-trace indices, both of cardinality
\(<\omega_2\).
Suppose that \(x,p_0\in2^{\omega_1}\) have reservoir support
contained in a set \(A_0\subseteq\omega_2\) of cardinality \(<\omega_2\).
Suppose further that for some \(c\in G_C\) and some
\(A_1\subseteq\omega_2\) of cardinality \(<\omega_2\), disjoint from \(A_0\),
\(c\) forces over the reservoir product that there is a clean sufficiently
allowable forcing whose auxiliary resources are supported by \(A_1\) and which
forces
\[
   M_m(x,p_0).
\]
Then, in the actual reservoir extension, there is a clean sufficiently allowable
forcing \(\mathbb Q^*\) which uses no container from \(B\), uses no reservoir
trace from \(R\), and satisfies
\[
   \mathbb Q^*\forces M_m(x,p_0).
\]
\end{lemma}

\begin{proof}
Apply Lemma~\ref{lem:cohen-reservoir-overspill} with the forbidden set enlarged
to include \(R\) and the reservoir supports of the names which decide the
container choices of the small presentation.
Since \(B\) has cardinality
\(<\omega_2\), the permutation can also be chosen so that the shifted container
choices avoid \(B\).
We obtain a permutation \(\pi\) of \(\omega_2\) such that
\[
   \widehat\pi(c)\in G_C,
   \qquad
   \pi\restriction A_0=\operatorname{id},
   \qquad
   \pi[A_1]\cap R=\emptyset,
\]
and such that the container choices in the shifted presentation avoid \(B\).
Since \(\widehat\pi\) is an automorphism of the reservoir product,
\(\widehat\pi(c)\) forces the \(\pi\)-shift of the statement forced by \(c\).
The
parameters \(x\) and \(p_0\) are fixed because their reservoir support is
contained in \(A_0\), and \(\pi\) fixes \(A_0\) pointwise.
Hence
\(\widehat\pi(c)\) forces that there is a clean sufficiently allowable forcing,
using the shifted auxiliary resources, which forces \(M_m(x,p_0)\).
These
shifted resources avoid \(B\) and \(R\) by construction.

Since \(\widehat\pi(c)\in G_C\), the actual reservoir extension contains such a
witness.
Call it \(\mathbb Q^*\). Then \(\mathbb Q^*\) is clean over the
forbidden resources and
\[
   \mathbb Q^*\forces M_m(x,p_0).
\]
\end{proof}

\begin{lemma}\label{lem:restriction-closure}
Let \(\mathbb R\) be a clean member of a sealed class
\(\Gamma^*_\rho(m,k,p)\), and let \(r\in\mathbb R\).
Then the restriction
\[
   \mathbb R\restriction r
   =
   \{s\in\mathbb R:s\leq r\}
\]
is again a clean member of \(\Gamma^*_\rho(m,k,p)\).
\end{lemma}

\begin{proof}
The restriction uses the same presentation and the same resources as
\(\mathbb R\).
It therefore avoids the same sealed containers, uses each
container only finitely often, and preserves all side-placement requirements.
\end{proof}

\begin{lemma}\label{lem:sigma11-upward}
If \(M_m(x,p)\) holds in an intermediate extension, then it remains true in every
further forcing extension preserving \(\omega_1\).
\end{lemma}

\begin{proof}
Write \(M_m(x,p)\) as a \(\Sigma^1_1\) assertion
\[
   \exists y\,\theta_m(x,p,y),
\]
where \(\theta_m\) is closed.
If \(y\) witnesses the assertion in the
intermediate model, then the same \(y\) remains present in every further
extension.
Since closed membership is absolute once \(\omega_1\) is preserved,
\(\theta_m(x,p,y)\) remains true.
\end{proof}

\section{Correctness of the construction}\label{sec:correctness-new}

\begin{lemma}\label{lem:neutralization-persists-new}
If a pair $(m,k,p)$ is neutralized at some stage, then in every later stage
\[
   \exists z\,\bigl(M_m(z,p)\wedge M_k(z,p)\bigr)
\]
holds.
\end{lemma}

\begin{proof}
The intersection statement is a $\Sigma^1_1$ assertion with actual witnesses
$z,y_0,y_1\in2^{\omega_1}$.
Once these witnesses exist, they remain witnesses
in every later forcing extension.
\end{proof}

\begin{lemma}\label{lem:no-bad-side-placements-new}
Suppose $(m,k,p)$ activates at stage $\rho$.
If at a later stage the construction codes
\[
   (\rho,x,m,k,p,0),
\]
then $M_k(x,p)$ is false in the final model.
If at a later stage the construction codes
\[
   (\rho,x,m,k,p,1),
\]
then $M_m(x,p)$ is false in the final model.
\end{lemma}

\begin{proof}
We first prove the side-zero case. Suppose that the construction codes
\[
   (\rho,x,m,k,p,0)
\]
at stage $\beta\geq\rho$.
By the definition of side zero, in $W[G_\beta]$
there is a $B_\beta$-clean member
\[
   \mathbb Q\in\Gamma^*_\rho(m,k,p)
\]
such that
\[
   \mathbb Q\forces M_m(x,p).
\]

Assume toward a contradiction that $M_k(x,p)$ holds in the final model.
By
Lemma~\ref{lem:bounded-appearance-new}, there is some $\gamma>\beta$ such that
the relevant witnesses already appear in $W[G_\gamma]$.
By the forcing theorem,
strengthening if necessary inside the actual quotient generic, there is a
condition
\[
   r\in \mathbb P_\gamma/G_\beta
\]
such that
\[
   r\forces_{\mathbb P_\gamma/G_\beta} M_k(x,p).
\]
Let
\[
   \mathbb R=(\mathbb P_\gamma/G_\beta)\restriction r.
\]
By Lemma~\ref{lem:tails-allowable-new} and Lemma~\ref{lem:restriction-closure},
$\mathbb R$ is a clean member of the sealed class
$\Gamma^*_\rho(m,k,p)$ over $W[G_\beta]$.
Now work over the activation model $W[G_\rho]$. Consider first the initial tail
\[
   \mathbb P_\beta/G_\rho,
\]
and then the restricted tail $\mathbb R$.
In the extension by these two factors,
the sets of containers and derived reservoir traces used before stage $\gamma$ are still
of size $<\omega_2$.
Choose reservoir supports for the parameters $x,p$ and a
reservoir condition in the fixed generic witnessing the statement that the
side-zero forcing exists.
Applying Lemma~\ref{lem:fresh-resource-relocation}
with
\[
   B=B_\gamma,
   \qquad
   R=R_\gamma,
\]
we obtain a fresh
\[
   \mathbb Q^*\in\Gamma^*_\rho(m,k,p)
\]
which avoids all containers and derived reservoir traces used before stage $\gamma$ and
satisfies
\[
   \mathbb Q^*\forces M_m(x,p).
\]

Thus the three-step forcing
\[
   (\mathbb P_\beta/G_\rho)
   *
   \mathbb R
   *
   \mathbb Q^*
\]
is a clean member of $\Gamma^*_\rho(m,k,p)$, by
Lemma~\ref{lem:concat-new}.
The middle factor $\mathbb R$ forces
$M_k(x,p)$. By Lemma~\ref{lem:sigma11-upward}, this remains true after the
further forcing with $\mathbb Q^*$.
The final factor $\mathbb Q^*$ forces
$M_m(x,p)$. Therefore the three-step forcing forces
\[
   M_m(x,p)\wedge M_k(x,p).
\]
In particular, it forces
\[
   \exists z\,\bigl(M_m(z,p)\wedge M_k(z,p)\bigr),
\]
with $z=x$.
This contradicts the activation of $(m,k,p)$ at stage $\rho$, since activation
says precisely that no clean member of $\Gamma^*_\rho(m,k,p)$ over $W[G_\rho]$
can force such an intersection.
Hence $M_k(x,p)$ is false in the final model.

We now prove the side-one case.
Suppose that the construction codes
\[
   (\rho,x,m,k,p,1)
\]
at stage $\beta\geq\rho$.
By the definition of the default side-one placement,
in $W[G_\beta]$ there is no $B_\beta$-clean member of
$\Gamma^*_\rho(m,k,p)$ forcing $M_m(x,p)$.
Assume toward a contradiction that $M_m(x,p)$ holds in the final model.
Again
by Lemma~\ref{lem:bounded-appearance-new} and the forcing theorem, there are
$\gamma>\beta$ and
\[
   r\in \mathbb P_\gamma/G_\beta
\]
such that
\[
   r\forces_{\mathbb P_\gamma/G_\beta} M_m(x,p).
\]
Let
\[
   \mathbb R=(\mathbb P_\gamma/G_\beta)\restriction r.
\]
By Lemma~\ref{lem:tails-allowable-new} and Lemma~\ref{lem:restriction-closure},
$\mathbb R$ is a $B_\beta$-clean member of $\Gamma^*_\rho(m,k,p)$.
But
$\mathbb R$ forces $M_m(x,p)$, contrary to the side-one choice made at stage
$\beta$. Therefore $M_m(x,p)$ is false in the final model.
\end{proof}

Let $(m,k,p)$ be a pair which is not neutralized in the final model, and let
$\rho$ be its activation stage.
The activation-stage sealed set $B_\rho$ is the
error term. Define
\[
   D_{m,k,p}=\{x\in2^{\omega_1}:\SepCoded^{B_\rho}_\rho(x,m,k,p,0)\}.
\]
Thus the separator ignores all containers which had already been used before the
pair activated.
\begin{lemma}[Total side assignment]\label{lem:total-side-assignment-new}
For every $x\in2^{\omega_1}$ in the final model, exactly one of
\[
   \SepCoded^{B_\rho}_\rho(x,m,k,p,0),
   \qquad
   \SepCoded^{B_\rho}_\rho(x,m,k,p,1)
\]
holds.
\end{lemma}

\begin{proof}
By bounded appearance, $x$ and $p$ occur in some bounded intermediate model
after activation.
At a later bookkeeping occurrence of $(x,p,m,k)$, the
construction assigns one of the two sides and directly codes the corresponding
tuple into a fresh container outside the current sealed set, hence outside
$B_\rho$.
Thus at least one clean side code exists, i.e., at least one of the
two displayed $\SepCoded^{B_\rho}_\rho$ relations holds.
By the no-unwanted-code
lemma, Lemma~\ref{lem:no-unwanted-direct-codes}, every counted clean code is
intentional. The construction records once a point has been assigned a side for
the activated pair and does nothing at later occurrences of the same point.
Hence two opposite intentional clean side codes for the same $x$ cannot be
produced.
\end{proof}

\begin{lemma}\label{lem:separator-complexity-new}
The set $D_{m,k,p}$ is $\Delta^1_1$.
\end{lemma}

\begin{proof}
The positive definition is the $\Sigma^1_1$ relation
$\SepCoded^{B_\rho}_\rho(x,m,k,p,0)$, with $B_\rho$ used as a boldface
parameter.
By total side assignment, the complement is defined by the
$\Sigma^1_1$ relation $\SepCoded^{B_\rho}_\rho(x,m,k,p,1)$. Hence
$D_{m,k,p}$ is both $\Sigma^1_1$ and $\Pi^1_1$.
\end{proof}

\begin{lemma}\label{lem:left-coverage-new}
If $(m,k,p)$ is not neutralized and $M_m(x,p)$ holds in the final model, then
$x\in D_{m,k,p}$.
\end{lemma}

\begin{proof}
By Lemma~\ref{lem:bounded-appearance-new}, the witnesses to $M_m(x,p)$ occur in
some bounded intermediate model.
At a later bookkeeping occurrence, the trivial
forcing is a clean member of the sealed class which forces $M_m(x,p)$.
Therefore
the construction assigns side zero to $x$ and writes the code into a fresh
container outside the current sealed set, in particular outside $B_\rho$.
\end{proof}

\begin{lemma}\label{lem:right-avoidance-new}
If $(m,k,p)$ is not neutralized, then
\[
   D_{m,k,p}\cap\{x:M_k(x,p)\}=\emptyset.
\]
\end{lemma}

\begin{proof}
If $x\in D_{m,k,p}$, then some container outside the error term $B_\rho$
witnesses $\SepCoded^{B_\rho}_\rho(x,m,k,p,0)$.
By
Lemma~\ref{lem:no-unwanted-direct-codes}, this side-zero code is intentional.
Since containers outside $B_\rho$ were not used before activation, the
intentional code was produced after the pair activated.
Lemma~\ref{lem:no-bad-side-placements-new} says that such a point never enters
the $k$-side.
\end{proof}

\begin{theorem}[Main theorem]\label{thm:main-separation}
In the final forcing extension, every two disjoint boldface $\Sig^1_1$ subsets
of $2^{\omega_1}$ are separated by a boldface $\Del^1_1$ set.
Moreover
\[
   2^\omega=\omega_1,
   \qquad
   2^{\omega_1}=\omega_2.
\]
\end{theorem}

\begin{proof}
The cardinal arithmetic was proved in
Lemma~\ref{lem:iteration-preservation-new}.
Let the two disjoint
$\Sigma^1_1$ sets be defined by $M_m(x,p)$ and $M_k(x,p)$.
If the pair
$(m,k,p)$ were neutralized, Lemma~\ref{lem:neutralization-persists-new} would
give an intersection in the final model.
Hence the pair is active at some
stage $\rho$. Form $D_{m,k,p}$ as above. By Lemma~\ref{lem:left-coverage-new},
the $m$-side is contained in $D_{m,k,p}$.
By
Lemma~\ref{lem:right-avoidance-new}, $D_{m,k,p}$ is disjoint from the $k$-side.
By Lemma~\ref{lem:separator-complexity-new}, $D_{m,k,p}$ is $\Delta^1_1$.
\end{proof}

\section{Open questions}\label{sec:open-questions}

\begin{question}
Can the direct $\omega_1$-length coding mechanism be strengthened to handle
reduction, not only separation?
\end{question}

\begin{question}
Does the separation model have consequences for the Borel$^*$ versus
$\Delta^1_1$ problem beyond the separation consequence proved here?
\end{question}

\begin{question}
Can the same direct-code architecture be lifted to regular $\kappa>
\omega_1$ satisfying $\kappa^{<\kappa}=\kappa$?
\end{question}

\bibliographystyle{plain}
\bibliography{omega1_sigma11_separation_references}

@incollection{AbrahamProperForcing,
  author    = {Abraham, Uri},
  title     = {Proper Forcing},
  booktitle = {Handbook of Set Theory},
  editor    = {Foreman, Matthew and Kanamori, Akihiro},
  publisher = {Springer},
  address   = {Dordrecht},
  year      = {2010},
  pages     = {333--394}
}

@article{AgostiniChapmanMottoRosPittonBorel,
  author        = {Agostini, Claudio and Chapman, Nick and Motto Ros, Luca and Pitton, Beatrice},
  title         = {Generalized {B}orel Sets},
  journal       = {arXiv preprint arXiv:2511.15663},
  year          = {2025},
  eprint        = {2511.15663},
  archivePrefix = {arXiv},
  primaryClass  = {math.LO}
}

@article{AgostiniMottoRosSchlichtPolish,
  author  = {Agostini, Claudio and Motto Ros, Luca and Schlicht, Philipp},
  title   = {Generalized {P}olish spaces at regular uncountable cardinals},
  journal = {Journal of the London Mathematical Society},
  volume  = {108},
  number  = {5},
  pages   = {1886--1929},
  year    = {2023},
  doi     = {10.1112/jlms.12797}
}

@article{BaumgartnerHarringtonKleinberg,
  author  = {Baumgartner, James E. and Harrington, Leo A. and Kleinberg, Eugene M.},
  title   = {Adding a closed unbounded set},
  journal = {Journal of Symbolic Logic},
  volume  = {41},
  number  = {2},
  pages   = {481--482},
  year    = {1976},
  doi     = {10.2307/2272248}
}

@article{DavidVeryAbsolute,
  author  = {David, Ren{\'e}},
  title   = {A very absolute {$\Pi^1_2$}-real singleton},
  journal = {Annals of Mathematical Logic},
  volume  = {23},
  pages   = {101--120},
  year    = {1982}
}

@article{FriedmanHoffelnerNS,
  author  = {Friedman, Sy-David and Hoffelner, Stefan},
  title   = {A {$\Sigma^1_4$} wellorder of the reals with {${\rm NS}_{\omega_1}$} saturated},
  journal = {Journal of Symbolic Logic},
  volume  = {84},
  number  = {4},
  pages   = {1466--1483},
  year    = {2019},
  doi     = {10.1017/jsl.2019.43}
}

@book{FriedmanHyttinenWeinstein,
  author    = {Friedman, Sy-David and Hyttinen, Tapani and Kulikov, Vadim},
  title     = {Generalized Descriptive Set Theory and Classification Theory},
  series    = {Memoirs of the American Mathematical Society},
  volume    = {230},
  publisher = {American Mathematical Society},
  year      = {2014},
  note      = {No. 1081}
}

@article{FriedmanKhomskiiKulikovRegularity,
  author  = {Friedman, Sy-David and Khomskii, Yurii and Kulikov, Vadim},
  title   = {Regularity properties on the generalized reals},
  journal = {Annals of Pure and Applied Logic},
  volume  = {167},
  number  = {4},
  pages   = {408--430},
  year    = {2016},
  doi     = {10.1016/j.apal.2016.01.001}
}

@article{FuchsHamkinsDegrees,
  author  = {Fuchs, Gunter and Hamkins, Joel David},
  title   = {Degrees of rigidity for {S}ouslin trees},
  journal = {Journal of Symbolic Logic},
  volume  = {74},
  number  = {2},
  pages   = {423--454},
  year    = {2009},
  doi     = {10.2178/jsl/1243948321}
}

@article{HoffelnerFailureReductionPresenceSeparation,
  author        = {Hoffelner, Stefan},
  title         = {A Failure of {$\Pi^1_{n+3}$}-Reduction in the Presence of {$\Sigma^1_{n+3}$}-Separation},
  journal       = {arXiv preprint arXiv:2312.02540},
  year          = {2023},
  eprint        = {2312.02540},
  archivePrefix = {arXiv},
  primaryClass  = {math.LO}
}

@article{HoffelnerForcingAxiomsUniformization,
  author  = {Hoffelner, Stefan},
  title   = {Forcing axioms and the uniformization-property},
  journal = {Annals of Pure and Applied Logic},
  volume  = {175},
  number  = {10},
  pages   = {103466},
  year    = {2024},
  doi     = {10.1016/j.apal.2024.103466}
}

@article{HoffelnerGlobalSigmaBPFA,
  author  = {Hoffelner, Stefan},
  title   = {The global {$\Sigma^1_{n+2}$}-uniformization property and {$\mathsf{BPFA}$}},
  journal = {Advances in Mathematics},
  volume  = {470},
  pages   = {110272},
  year    = {2025},
  doi     = {10.1016/j.aim.2025.110272}
}

@article{HoffelnerLarsonSchindlerWu,
  author  = {Hoffelner, Stefan and Larson, Paul B. and Schindler, Ralf and Wu, Liuzhen},
  title   = {Forcing axioms and the definability of the nonstationary ideal on the first uncountable},
  journal = {Journal of Symbolic Logic},
  volume  = {89},
  number  = {4},
  pages   = {1641--1658},
  year    = {2024},
  doi     = {10.1017/jsl.2023.40}
}

@article{HoffelnerMAIFailureSeparation,
  author  = {Hoffelner, Stefan},
  title   = {{$\mathsf{MA}(\mathcal I)$} and a failure of separation on the third level},
  journal = {Annals of Pure and Applied Logic},
  volume  = {177},
  number  = {3},
  pages   = {103667},
  year    = {2026},
  doi     = {10.1016/j.apal.2025.103667}
}

@article{HoffelnerNSDeltaOne,
  author  = {Hoffelner, Stefan},
  title   = {{${\rm NS}_{\omega_1}$} saturated and {$\Delta_1$}-definable},
  journal = {Journal of Symbolic Logic},
  volume  = {86},
  number  = {1},
  pages   = {25--59},
  year    = {2021},
  doi     = {10.1017/jsl.2021.23}
}

@article{HoffelnerPi13Reduction,
  author  = {Hoffelner, Stefan},
  title   = {Forcing the {$\Pi^1_3$}-reduction property and a failure of {$\Pi^1_3$}-uniformization},
  journal = {Annals of Pure and Applied Logic},
  volume  = {174},
  number  = {8},
  pages   = {103292},
  year    = {2023},
  doi     = {10.1016/j.apal.2023.103292}
}

@article{HoffelnerPi13UniformizationWellorder,
  author        = {Hoffelner, Stefan},
  title         = {A Universe with a {$\Delta^1_n$}-definable well-order of the reals, {$\mathsf{CH}$} and {$\Pi^1_n$}-Uniformization},
  journal       = {arXiv preprint arXiv:2506.21778},
  year          = {2025},
  eprint        = {2506.21778},
  archivePrefix = {arXiv},
  primaryClass  = {math.LO}
}

@article{HoffelnerPi1nUniformization,
  author        = {Hoffelner, Stefan},
  title         = {Forcing the {$\Pi^1_n$}-Uniformization Property},
  journal       = {arXiv preprint arXiv:2103.11748},
  year          = {2021},
  eprint        = {2103.11748},
  archivePrefix = {arXiv},
  primaryClass  = {math.LO}
}

@article{HoffelnerSigma13Separation,
  author  = {Hoffelner, Stefan},
  title   = {Forcing the {$\Sigma^1_3$}-separation property},
  journal = {Journal of Mathematical Logic},
  volume  = {22},
  number  = {2},
  pages   = {2250008},
  year    = {2022},
  doi     = {10.1142/S0219061322500088}
}

@article{HoffelnerLargeContinuumGlobalSigmaWellorder,
  author        = {Hoffelner, Stefan},
  title         = {A Universe with large Continuum, global {$\Sigma$}-Uniformization and a projective Well-Order of its Reals},
  journal       = {arXiv preprint arXiv:2506.12393},
  year          = {2025},
  eprint        = {2506.12393},
  archivePrefix = {arXiv},
  primaryClass  = {math.LO}
}

@article{HoffelnerUpperSigmaUniformization,
  author        = {Hoffelner, Stefan},
  title         = {Forcing upper {$\Sigma$}-uniformization in the presence of lower {$\Pi$}-reduction or uniformization},
  journal       = {arXiv preprint arXiv:2511.05081},
  year          = {2025},
  eprint        = {2511.05081},
  archivePrefix = {arXiv},
  primaryClass  = {math.LO}
}

@article{HoffelnerSigma13Sigma14Uniformization,
  author        = {Hoffelner, Stefan},
  title         = {On {$\boldsymbol{\Sigma}^1_3$}- and {$\Sigma^1_4$}-uniformization},
  journal       = {arXiv preprint arXiv:2604.19360},
  year          = {2026},
  eprint        = {2604.19360},
  archivePrefix = {arXiv},
  primaryClass  = {math.LO}
}

@misc{HoffelnerDelfinoLocal,
  author = {Hoffelner, Stefan},
  title  = {On a local variant of the 12th {D}elfino problem},
  year   = {2025},
  note   = {Preprint}
}

@misc{HoffelnerDelfinoLocalII,
  author = {Hoffelner, Stefan},
  title  = {Coding and {$\infty$}-allowability for the local {D}elfino construction},
  year   = {2025},
  note   = {Preprint}
}

@incollection{HyttinenKulikovBorelStar,
  author    = {Hyttinen, Tapani and Kulikov, Vadim},
  title     = {{Borel$^*$} Sets in the Generalized {B}aire Space and Infinitary Languages},
  booktitle = {Jaakko Hintikka on Knowledge and Game-Theoretical Semantics},
  editor    = {van Ditmarsch, Hans and Sandu, Gabriel},
  series    = {Outstanding Contributions to Logic},
  volume    = {12},
  publisher = {Springer},
  address   = {Cham},
  pages     = {395--412},
  year      = {2018},
  doi       = {10.1007/978-3-319-62864-6_16}
}

@book{JechSetTheory,
  author    = {Jech, Thomas},
  title     = {Set Theory},
  series    = {Springer Monographs in Mathematics},
  edition   = {Third Millennium},
  publisher = {Springer},
  year      = {2003}
}

@incollection{JensenSolovay,
  author    = {Jensen, Ronald B. and Solovay, Robert M.},
  title     = {Some applications of almost disjoint sets},
  booktitle = {Mathematical Logic and Foundations of Set Theory},
  series    = {Studies in Logic and the Foundations of Mathematics},
  volume    = {59},
  publisher = {North-Holland},
  pages     = {84--104},
  year      = {1970}
}

@book{KechrisClassicalDST,
  author    = {Kechris, Alexander S.},
  title     = {Classical Descriptive Set Theory},
  series    = {Graduate Texts in Mathematics},
  volume    = {156},
  publisher = {Springer},
  year      = {1995}
}

@article{KhomskiiLaguzziLoeweSharankouQuestions,
  author  = {Khomskii, Yurii and Laguzzi, Giorgio and L{\"o}we, Benedikt and Sharankou, Ilya},
  title   = {Questions on generalised {B}aire spaces},
  journal = {Mathematical Logic Quarterly},
  volume  = {62},
  number  = {4--5},
  pages   = {439--456},
  year    = {2016},
  doi     = {10.1002/malq.201600051}
}

@article{LueckeMottoRosSchlichtHurewicz,
  author  = {L{\"u}cke, Philipp and Motto Ros, Luca and Schlicht, Philipp},
  title   = {The {H}urewicz dichotomy for generalized {B}aire spaces},
  journal = {Israel Journal of Mathematics},
  volume  = {216},
  number  = {2},
  pages   = {973--1022},
  year    = {2016},
  doi     = {10.1007/s11856-016-1435-1}
}

@article{LueckeSchlichtContinuousImages,
  author  = {L{\"u}cke, Philipp and Schlicht, Philipp},
  title   = {Continuous images of closed sets in generalized {B}aire spaces},
  journal = {Israel Journal of Mathematics},
  volume  = {209},
  number  = {1},
  pages   = {421--461},
  year    = {2015},
  doi     = {10.1007/s11856-015-1224-2}
}

@article{LueckeSigma11Definability,
  author  = {L{\"u}cke, Philipp},
  title   = {{$\Sigma^1_1$}-definability at uncountable regular cardinals},
  journal = {Journal of Symbolic Logic},
  volume  = {77},
  number  = {3},
  pages   = {1011--1046},
  year    = {2012}
}

@article{MeklerVaananen,
  author  = {Mekler, Alan and V{\"a}{\"a}n{\"a}nen, Jouko},
  title   = {Trees and {$\Pi^1_1$}-subsets of {${}^{\omega_1}\omega_1$}},
  journal = {Journal of Symbolic Logic},
  volume  = {58},
  number  = {3},
  pages   = {1052--1070},
  year    = {1993}
}

@article{MiyamotoSouslinCSIterations,
  author  = {Miyamoto, Tadatoshi},
  title   = {{$\omega_1$}-{S}ouslin trees under countable support iterations},
  journal = {Fundamenta Mathematicae},
  volume  = {142},
  number  = {3},
  pages   = {257--261},
  year    = {1993}
}

@book{MoschovakisDST,
  author    = {Moschovakis, Yiannis N.},
  title     = {Descriptive Set Theory},
  series    = {Mathematical Surveys and Monographs},
  volume    = {155},
  publisher = {American Mathematical Society},
  year      = {2009}
}

\end{document}